\documentclass[final,1p,times,12pt]{elsarticle}

\usepackage{amssymb}

\usepackage{graphicx}
\usepackage{amsmath}

\usepackage{amssymb}

\usepackage{amsthm}
\usepackage{bbm,comment}
\usepackage{dsfont}
\usepackage{bm}

\usepackage{amsfonts}
\usepackage{amssymb}
\usepackage{esint}
\usepackage{graphicx}

\usepackage{enumerate}
\usepackage{verbatim}
\usepackage{hyperref}
\usepackage[english]{babel}
\usepackage{bbm}
\usepackage{dsfont}
\usepackage{bm}
\usepackage{mathrsfs}
\usepackage{relsize}
\usepackage{textcomp}

\usepackage{color}

\usepackage{subfigure}
\usepackage{indentfirst}

\usepackage{xcolor}
\definecolor{lucagreen}{RGB}{0,150,0}

\usepackage{geometry}

\DeclareMathAlphabet      {\mathbf}{OT1}{cmr}{bx}{n}

\newtheorem{theorem}{Theorem}[section]
\newtheorem{proposition}[theorem]{Proposition}
\newtheorem{lemma}[theorem]{Lemma}
\newtheorem{corollary}[theorem]{Corollary}

\newtheorem{definition}[theorem]{Definition} 
\newtheorem{remark}[theorem]{Remark}         
\newtheorem{hp}[theorem]{Hypotheses}         

\def \R{\mathbb{R}}

\newcommand{\miezz}{\frac{1}{2}}

\newcommand{\eps}{\varepsilon}

\newcommand{\into}{\ensuremath{\int_{\Omega}}}

\newcommand{\inti}{\ensuremath{\int_{0}^{t}\int_{\Omega}}}

\newcommand{\intif}{\ensuremath{\int_{0}^{T}\int_{\Omega}}}

\newcommand{\norm}[1]{\ensuremath{\left\Arrowvert #1 \right\Arrowvert}}
\newcommand{\norminf}[1]{\ensuremath{\left\Arrowvert #1 \right\Arrowvert_\infty}}

\newcommand{\spazio}{\hspace{0.08cm}}

\newcommand{\dm}[1]{\ensuremath{\frac{\delta #1}{\delta m}}}
\newcommand{\dw}{\mathbf{d}_1}
\newcommand{\supt}{\sup\limits_{t\in[t_0,T]}}
\newcommand{\supo}{\sup\limits_{t\in[0,T]}}
\newcommand{\tr}[1]{\mathrm{tr}(a(x)D^2#1)}

\newcommand{\bdone}[1]{{#1}_{\vert \partial\Omega}=0}
\newcommand{\am}{_{\frac{\alpha}{2},\alpha}}
\newcommand{\amu}{_{\frac{1+\alpha}{2},1+\alpha}}
\newcommand{\amd}{_{1+\frac{\alpha}{2},2+\alpha}}
\newcommand{\amv}{_{1,2+\alpha}}

\newcommand{\be}{\begin{equation}}
\newcommand{\ee}{\end{equation}}
\newcommand{\amf}{_{L^1}}
\newcommand{\amc}{_{-(1+\alpha),D}}

\newcommand{\amb}{_{-(1+\alpha)}}
\newcommand{\aam}{_{-\left(\frac\alpha2,\alpha\right)}}

\newcommand{\amm}[1]{_{#1,-(1+\alpha),D}}
\newcommand{\diric}[1]{#1_{|\partial\Omega=0}}

\newcommand{\nobt}[1]{\sum\limits_{ij}\partial^2_{ij}(a_{ij}(x) #1)}

\journal{Nonlinear Analysis}

\begin{document}

\begin{frontmatter}



\title{The Master Equation in a Bounded Domain with Absorption}


\author[label1]{Luca Di Persio}

\author[label2]{Matteo Garbelli}

\author[label3]{Michele Ricciardi}

\address[label1]{Department of Computer Science, University of Verona, 
Strada le Grazie $15$, Verona, $37134$, Italy, 
{\it luca.dipersio@univr.it}}

\address[label2]{Department of Computer Science, University of Verona, 
Strada le Grazie $15$, Verona, $37134$, Italy, 
{\it matteo.garbelli@univr.it}    }

\address[label3]{Department of Economics, LUISS - Guido Carli University of Rome, Viale Romania 32, Roma, 00197, Italy,
{\it ricciardim@luiss.it}}


%
%

\begin{abstract}
We analyze the Master Equation within Mean Field Games (MFG) theory considering a bounded domain with homogeneous Dirichlet conditions. Concerning the N-players differential game, the player's dynamic ends when touching the boundary. We analyze the well-posedness of the Master Equation and the regularity of its solutions for a suitable class of parabolic equations.
\end{abstract}




\end{frontmatter}



\section{Introduction}
Mean Field Games (MFG) theory was introduced in 2006 by J.-M. Lasry and P.-L. Lions (\cite{LL1,LL2,LL-japan,LL3}) to describe the asymptotic behavior of differential games with a large number of players, also called agents. A similar definition was given in the same year by Caines, Huang and Malham\'e \cite{HCM}. 

In the description of a game with $N$ players, each agent chooses his own strategy to minimize a certain functional cost, which also depends on the strategies of the other agents. The structure of this problem becomes really intricate when the number of players is large, and this leads us to study the asymptotic behavior of the Nash equilibria for $N\gg1$, under some symmetry conditions for the agents and their dynamics. The system introduced to describe the macroscopic behavior of the optimal strategies for $N\to+\infty$ is called \emph{Mean Field Games system}, and it is composed of a backward Hamilton-Jacobi-Bellman equation for the value function $u$ of the single agent coupled with a forward Fokker-Planck equation satisfied by $m$, the distribution law of the population. A typical form of this system is the following:
\begin{equation}\label{eq:MfG}
\begin{cases}
-u_t-\mathrm{tr}(a(x)D^2u)+H(t,x,Du)=F(t,x,m),\\
m_t-\sum\limits_{i,j}\partial^2_{i,j}(a_{ij}(x)m)-\mathrm{div}(mH_p(t,x,Du))=0,\\
m(0)=m_0,\qquad u(T)=G(x,m(T)).
\end{cases}
\end{equation}
Here, $H$ is called the \emph{Hamiltonian} of the system, whereas $a$ is a uniformly elliptic matrix, which corresponds to the second-order part of the infinitesimal generator associated with the SDE
\begin{equation}
\label{eq:sde1}
dX_t=b(t,X_t,\alpha_t)dt+\sqrt 2\sigma(X_t)dB_t,
\end{equation}
which represents the controlled dynamic of the generic player.


Lions proved, in his lectures at Collège de France (some lectures are still available in \cite{Lionscollege}), that the solutions of the MFG system are the trajectories of an infinite-dimensional PDE, called \emph{Master Equation}, which summarizes the MFG system in a unique equation. The importance of the Master Equation relies on the so-called \emph{convergence problem}: a rigorous convergence result of the Nash equilibria in the $N$-players game towards the optimal strategies in the MFG presents many difficulties, due to the lack of compactness properties of \eqref{eq:MfG}. Hence, most of the results in the literature related to the convergence problem are proved by means of the Master Equation.

So far, the Master Equation and the convergence problem have been studied in various frameworks. However, most of the literature mainly considers the case where the players' state space is the torus $\mathbb{T}^d$, hence a periodic framework, or the whole space $\R^d$.

Nevertheless, in many applications it is crucial to consider situations where the dynamic of the players is confined in a bounded domain, and boundary conditions become unavoidable issues we have to deal with. See, for example, the models developed in \cite{aid1}, \cite{burger}, \cite{Burzoni}, \cite{Campi}, \cite{GBT}.

In this setting, several different boundary behaviors can arise depending on how the controlled dynamics
\eqref{eq:sde1} and the Mean Field Games system are modeled. The most common situations can be summarized as follows:
\begin{itemize}
\item[(i)] \emph{Invariant domains.}  
The drift and diffusion coefficients in the dynamics \eqref{eq:sde1} are constructed in such a way that the trajectories of the agents remain inside the domain for any admissible control. 
This setting leads to Mean Field Games with \emph{invariance conditions} on the boundary, which have been studied, e.g., in \cite{PorrettaRicciardi}.

\item[(ii)] \emph{State constraints.}  
Not all the controls guarantee the required confinement, but the optimal control is chosen in such a way that the optimal trajectories stay inside the domain. 
This situation is referred to as Mean Field Games with \emph{state constraints}. A comprehensive result of first-order Mean Field Games (without diffusion $\sigma$ in \eqref{eq:sde1}) has been studied in \cite{MR3888967,CaCaCa2018}. See also \cite{MR4135073,cannarsa2020weak,CAPUANI2022180,CapMarRic} for the first-order case, and \cite{PorRic-ergodic} for a result in the second-order case.

\item[(iii)] \emph{Reflecting dynamics.}  
In this case, the process $X_t$ is forced to remain in the domain by means of a reflection term acting on the boundary (which may represent, in the applications, a physical wall or barrier). 
The associated Mean Field Games system is then coupled with \emph{Neumann boundary conditions} for both the HJB and the Fokker–Planck equations.

\item[(iv)] \emph{Absorbing boundaries.}  
In this last case, the process can exit the domain, but the game stops as soon as the trajectory hits the boundary. 
Here, the Mean Field Games system \eqref{eq:MfG} is complemented by \emph{Dirichlet boundary conditions}, corresponding to the \emph{absorption} of the players at the boundary.
\end{itemize}

The last two cases have been widely studied for the MFG system. See, for example, \cite{Porretta} for a very comprehensive result of Dirichlet and Neumann boundary conditions with minimal assumptions on the data, \cite{FestaRicciardi} for forward-forward Mean Field Games with Dirichlet or Neumann Boundary conditions, and \cite{Cirant} for Neumann boundary conditions in multi-population MFG. See also \cite{GomesRicciardi} for a result on first-order Mean Field Games with Neumann boundary conditions.

As regards the Master Equation, the results in a bounded domain are significantly fewer. The papers \cite{Ricciardi,Ricciardi3} provide a well-posedness result on the Master Equation and the study of the convergence problem in the case of Neumann boundary conditions, but we are not aware of results related to the other scenarios previously described.

In our paper, we aim to study the well-posedness of the Master Equation in a bounded domain $\Omega\subset\R^d$ and in a framework of absorbing boundary conditions. This paper can be thought of as a first, crucial step in order to prove the related convergence problem for $N$-player games with absorbing boundary conditions, which will be presented in a forthcoming paper.

In this case, the Master Equation is defined from its characteristics, which are solutions of the MFG system \eqref{meanfieldgames} constrained with the following Dirichlet boundary conditions: for all $x\in\partial\Omega$ and $t\in[0,T]$,
\begin{equation}\label{eq:Dirichlet}
u(t,x)=0,\qquad\qquad\qquad m(t,x)=0.
\end{equation}
To be more precise, we consider the set $\mathcal{P}^{sub}(\Omega)$ of Borel sub-probability measures on $\Omega$, and for $(t_0,x,m_0)\in (0,T)\times\Omega\times\mathcal{P}^{sub}(\Omega)$ we consider the solution $(u,m)$ of the MFG system \eqref{eq:MfG}, with boundary conditions \eqref{eq:Dirichlet} and with initial condition $m(t_0)=m_0$. Then we define the function $U$ as
\begin{equation}
\label{eq:u2}
U:[0,T]\times\Omega\times\mathcal{P}^{sub}(\Omega),\qquad U(t_0,x,m_0):=u(t_0,x).
\end{equation}
If we compute, at least formally, the equation satisfied by $U$, we obtain an infinite-dimensional equation, called {\bf Master Equation}. In this case, the Master Equation takes the following form:
\begin{equation}\begin{split}
\label{Master}
\left\{
\begin{array}{rl}
\displaystyle
&-\,\partial_t U(t,x,m)-\mathrm{tr}\left(a(x)D_x^2 U(t,x,m)\right)+H\left(t,x,D_x U(t,x,m)\right) \medskip \\
\displaystyle & \displaystyle \quad -\int_\Omega \mathrm{tr}\left(a(y)D_y D_m U(t,x,m,y)\right)dm(y) \medskip \\ 
\displaystyle
& \displaystyle	\quad + \int_\Omega D_m U(t,x,m,y)\cdot H_p(t,y,D_x U(t,y,m))dm(y)= F(t,x,m)  \medskip \\
\displaystyle &\hspace{4.8cm}\mbox{in }[0,T]\times\Omega\times\mathcal{P}^{sub}(\Omega)\spazio, \medskip\\
\displaystyle	&U(T,x,m)=G(x,m)\hspace{1.2cm}\mbox{in }\Omega\times\mathcal{P}^{sub}(\Omega)\spazio, \medskip\\
\displaystyle	& \displaystyle U(t,x,m)=0\hspace{1.8cm}\quad\,\,\mbox{for }(t,x,m)\in[0,T]\times\partial\Omega\times\mathcal{P}^{sub}(\Omega)\,,  \medskip \\
& \displaystyle\dm{U}(t,x,m,y)=0\hspace{1.8cm}\mbox{for }(t,x,m,y)\in[0,T]\times\Omega\times\mathcal{P}^{sub}(\Omega)\times\partial\Omega\,,
\end{array}
\right.
\end{split}\end{equation}
where $D_mU$ and $\dm{U}$ are two derivatives, considered with respect to the measure variable $m$, whose precise definition will be given later.

In the literature, an equation like \eqref{Master} is often called the \emph{First order Master Equation}, or \emph{Master Equation without common noise}. 
This terminology reflects the fact that, in the underlying mean field model, the representative agent evolves according to the dynamics \eqref{eq:sde1}, where the stochastic term $dB_t$ accounts for an \emph{idiosyncratic noise}, that is, a source of randomness affecting each individual independently of the others. 
In contrast, when a common source of uncertainty $dW_t$, shared by all agents, is included in the dynamics, one obtains the so-called \emph{Second order Master Equation}, or \emph{Master Equation with common noise}, which involves additional second-order terms in the measure variable.

After being introduced in \cite{LL-japan}, the Master Equation has been studied in many papers, almost always in the periodic case $\Omega=\mathbb{T}^d$, or in the whole space $\Omega=\R^d$ (except for the already cited articles \cite{Ricciardi,Ricciardi3}). In \cite{Buck}, Buckdahn, Li, and Peng proved the well-posedness of the first-order Master equation without coupling terms by means of probabilistic arguments, while  Chassagneux, Crisan, and Delarue, in \cite{28},  provided a first exhaustive existence and uniqueness result of solution, still without common noise. Moreover, Gangbo and Swiech, in \cite{nuova16}, gave a short-time existence for the Master Equation in the presence of common noise.

The most relevant result was certainly given in \cite{card}, where Cardaliaguet, Delarue, Lasry, and Lions proved the well-posedness of the Master Equation, with and without common noise, in a periodic setting.

Moreover,
Carmona and Delarue in \cite{Carmona}
derived convergence results in the whole space, while in  \cite{nuova4,ramadan},
Delarue, Lacker and Ramanan used the Master Equation to analyse the large deviation problem as well as the central limit theorem.
Concerning the major-minor problem, Cardaliaguet, Cirant and Porretta studied a convergence result in \cite{CCP}. For finite state problems, we refer to the works of Bayraktar and Cohen in \cite{nuova1} and by Cecchin and Pelino in \cite{nuova11}. See also the work of Bertucci, Lasry and Lions \cite{BLL2}, who studied the Master Equation for the Finite State Space Planning Problem. More recently, in the setting of MFG with exhaustible resources, i.e. the continuum limit of a dynamic game of exhaustible resources modelling Cournot competition between producers, see e.g. \cite{graber1, graber2} for more details, Graber and Sircar studied the Master Equation for a MFG
of controls with absorption in \cite{graber_me}. Differently from the classical setting studied in this paper where the interaction is through the mean of the state variable, in \cite{graber_me} 
the authors studied a model of Cournot competition between producers whose states depend on the empirical measure over controls.

Other important papers about the Master Equation and the convergence problem are given by \cite{dybala, sonocarmela,checchino,cicciocaputo,jedi,chicazzosiete, fifa21,tonali,jakob,loacker,loacker2,koulibaly,durr}.

Concerning the model with absorption, the MFG system has already been studied in the literature. See, for example, \cite{Campi}, where the authors studied the $ N$-player game and the MFG by means of probabilistic arguments, working directly on the SDEs describing the dynamics and introducing renormalized empirical measures to take care of the loss of mass due to the players leaving the domain. Those results are also used to prove that the optimal strategies of the MFG provide $\eps-$approximated Nash equilibria in the $N$-players game. Some generalizations of the described results were given in \cite{Burzoni}, where a model of bank run is considered within the framework of an MFG model with absorption and common noise, and in \cite{campighiolivieri}, where a dependence on past absorptions is allowed. See also \cite{Diogo1} for an elliptic first-order case.
Similar problems, such as minimal-time MFGs where agents want to leave a given bounded domain through its boundary (or a part of it) in minimal time, with application to crowd motion, have been studied in \cite{MS}. Similarly, in \cite{Bertucci,Bertucci2,BDT}, optimal stopping MFGs were studied for a model where a representative agent chooses both the optimal control and the optimal time to exit the game.
\bigskip

As already mentioned, the main goal of this paper is to establish the well-posedness of the first-order Master Equation \eqref{Master} in a bounded domain with absorbing (Dirichlet) boundary conditions. 
To the best of our knowledge, this is the first existence result for the Master Equation in such a framework. 
The natural setting for this problem is the space $\mathcal{P}^{sub}(\Omega)$ of sub-probability measures, which provides the appropriate structure to describe the dissipation of mass caused by absorption at the boundary.

First of all, we give some comments about the boundary conditions appearing in \eqref{Master}. While the first condition $U(t,x,m)=0$ for $x\in\partial\Omega$ is a natural consequence of the Dirichlet condition \eqref{eq:Dirichlet} for the value function $u$, the condition
\begin{equation}
\frac{\delta U}{\delta m} = 0 \quad \text{on } [0,T)\times \mathcal{P}^{sub}(\Omega)\times \partial\Omega \, , 
\end{equation}
requires an additional comment. While this condition has recently appeared in the mean field control literature with absorption in \cite{card-abs}, its formulation in the MFG context is, as far as we know, completely new in the literature. It relies on the fact that we have to take care not only of the boundary conditions for the state variable $x\in\Omega$, but also for the measure variable $m\in\mathcal P^{sub}(\Omega)$.

The two boundary conditions can be better understood in the $N$-players game. Loosely speaking, assume that $U$ is a smooth function and that it can be extended in the set $[0,T]\times\overline\Omega\times\mathcal P(\overline\Omega)$. Then the boundary condition in the space variable means that, if the game for a generic player starts at $\bar x\in\partial\Omega$, the game is immediately stopped and the value function (provided a vanishing condition of the final cost $G$ at the boundary) vanishes.

For the measure variable, the boundary condition in this case implies that
\begin{equation}
U(t,x,m)=U(t,x,\tilde m)\qquad\mbox{if }m(\Omega)=\tilde m(\Omega).
\end{equation}
In particular, $U(t,x,\delta_y)=U(t,x,\delta_{\tilde y})$ if $y,\tilde y\in\partial\Omega$. Hence, the value function for a player remains unchanged if the players at the boundary change their position. This is quite natural, since the agents, once they reach the boundary, play no role in the game for the other players.

More details on the meaning of these boundary conditions in the $N$-players game are given in the already cited \cite{card-abs} for the mean-field control, and will be given in our forthcoming paper on the convergence problem for MFG with absorption.

We also want to underline that the existence proof for the Master Equation relies on several estimates for the associated Mean Field Games system, and in particular on a careful study of a Fokker–Planck equation with Dirichlet boundary conditions in a very general form. 
As a byproduct, we refine some regularity estimates for such equations that may be of independent interest.
\bigskip

The function $U$ is defined as in \eqref{eq:u2}, and various estimates, such as global bounds and global Lipschitz regularity, are established. Notably, one of the key challenges in demonstrating that $U$ satisfies \eqref{Master} is establishing its $\mathcal{C}^1$ continuity with respect to $m$. This step necessitates a meticulous analysis of the linearized mean field game system (refer to \cite{card,Ricciardi}) to establish strong regularity of $U$ in both the spatial and measure variables.

It is worth highlighting that these regularity estimates demand strong regularity not only in the spatial domain but also in the measure variable.

In the spatial domain, the regularity is derived in \cite{card} through differentiation of the equation with respect to $x$. However, in the case of Dirichlet boundary conditions, as well as in general for any boundary conditions, such methods are not directly applicable. Instead, we obtain these bounds by employing a distinct set of space-time estimates that require careful handling.

Moreover, it is essential to note that regularity estimates for the Dirichlet parabolic equation necessitate compatibility conditions between initial and boundary data. Unfortunately, these compatibility conditions cannot always be guaranteed within this context. Consequently, we extend the estimates obtained in \cite{card} through an in-depth investigation of the regularity of solutions to the Fokker-Planck equation.

A noteworthy novelty in this work is the meticulous study of the Fokker-Planck equation and the linearized Mean Field Games system. This study is conducted within negative-order Hölder spaces, and the well-posedness is established via approximation using smooth functions. However, it is important to emphasize that smooth functions are not dense in the dual of Hölder spaces, as pointed out in \cite{jakob}. As a result, we must confine our analysis to appropriate subsets of these dual spaces. Notably, these ideas can also be applied to fix the issues present in the Neumann case \cite{Ricciardi}.

The article is organized as follows: in Section \ref{sec2}, we introduce the basic notations and hypotheses we will need throughout the article, and we state the two main results we want to prove. In Section \ref{sec3} we give a formal proof for the well-posedness of the Master Equation, which becomes rigorous provided some strong regularity results of the function $U$. In Sections \ref{sec4} and \ref{sec5} we state the latter regularity results for $U$: in particular, Section \ref{sec4} is devoted to the study of the Fokker-Planck equation and the Mean Field Games, whereas in Section \ref{sec5} we analyze a MFG linearized system which provides a suitable regularity for the function $U$. 

\section{Notation and assumptions}
\label{sec2}

Let $T>0$ and $\Omega\subset\R^d$ be the closure of an open, bounded and connected set, with the boundary of class $\mathcal{C}^{2+\alpha}$, for some $\alpha>0$; we denote with $Q_T$  the set $Q_T:=[0,T]\times\Omega$.
We call $d(\cdot)$ the oriented distance function from the boundary of $\Omega$, defined as
\begin{equation}
d(x)=\begin{cases}
\,\,\,\ \mathrm{dist}(x,\partial\Omega) & x\in\Omega\,,\\
-\mathrm{dist}(x,\partial\Omega) & x\notin\Omega\,.
\end{cases}
\end{equation}
Thanks to \cite{DistFunct}, we have $d(\cdot)\in\mathcal C^{2+\alpha}$ in a neighbourhood of the boundary. Since we are only interested in the local character of $d$ near $\partial\Omega$, when we write $d(\cdot)$ in this paper, we mean a $\mathcal C^{2+\alpha}$ function coinciding with $d$ in a neighbourhood of the boundary. 


\subsection{H\"{o}lder spaces}
In what follows, we briefly recall basic notions about the Banach function spaces used throughout the paper, see, e.g., \cite{lsu,Ricciardi}, for more details.

For $n\ge0$ and $\alpha\in(0,1)$,  $\mathcal{C}^{n+\alpha}(\Omega)$ is defined as the space of functions $n$-times differentiable, with derivatives $\alpha$-H\"{o}lder continuous. 
Being $\phi \in \mathcal{C}^{n+\alpha}(\Omega)$, its norm is defined in the following way:
\begin{align*}
\norm{\phi}_{n+\alpha}:=\sum\limits_{\vert \ell\vert \le n}\norminf{D^l\phi}+\sum\limits_{\vert \ell\vert = n}\sup\limits_{x\neq y}\frac{\vert D^\ell\phi(x)-D^\ell\phi(y)\vert }{\vert x-y\vert ^\alpha}\,.
\end{align*}

Similarly, the parabolic space $\mathcal{C}^{\frac{n+\alpha}{2},n+\alpha}(Q_T)$ consists of functions $\phi$ admitting derivatives $D_t^rD_x^s\phi$, with $2r+s\le n$, and with norm
\begin{align*}
\norm{\phi}_{\frac{n+\alpha}{2},n+\alpha}:= \sum\limits_{2r+s\le n} & \norminf{D_t^rD^s_x\phi}+\sum\limits_{2r+s=n}\sup\limits_t\norm{D_t^rD_x^s\phi(t,\cdot)}_{\alpha}\\&+\sum\limits_{0<n+\alpha-2r-s<2}\sup\limits_x\norm{D_t^rD_x^s\phi(\cdot,x)}_{\frac{n+\alpha-2r-s}{2}}\,.
\end{align*}

In order to work with Dirichlet boundary conditions, we define the spaces $\mathcal{C}^{n+\alpha,D}(\Omega)$ and $\mathcal{C}^{\frac{n+\alpha}{2},n+\alpha,D}(Q_T)$ as the subspaces of functions belonging to $\mathcal{C}^{n+\alpha}(\Omega)$, resp. to $\mathcal{C}^{\frac{n+\alpha}{2},n+\alpha}(Q_T)$,  vanishing at $\partial\Omega$. When there is no risk of confusion, we will omit the set $\Omega$ or $Q_T$.

Analogously, we define the spaces $\mathcal{C}^{0,\alpha}$, $\mathcal{C}^{\alpha,0}$, $\mathcal{C}^{1,2+\alpha}$. 
For the sake of completeness, also because of its relevance throughout the paper, let us specify the norm equipping the latter functional space:
$$
\norm{\phi}_{1,2+\alpha}:=\norminf{\phi}+\norm{\phi_t}_{0,\alpha}+\norminf{D_x\phi}+\norm{D^2_x\phi}_{0,\alpha}\,.
$$

Eventually, we have to work with suitable subsets of the dual spaces of $\mathcal{C}^{n+\alpha}$ and  $\mathcal{C}^{n+\alpha,D}$. In particular, we denote as $\mathcal{C}^{-(n+\alpha)}$ the set
\[
\mathcal{C}^{-(n+\alpha)}:=\left\{ \rho\in\left(\mathcal C^{(n+\alpha)}\right)'\,\Big|\,\langle\rho,\phi\rangle=\sum\limits_{|\gamma|\le n}\into\partial^\gamma\phi(x)\,\rho_\gamma(dx)\quad\forall\phi\in\mathcal C^{n+\alpha}\,,\quad {(\rho_\gamma)}_\gamma\mbox{ measures on }\Omega\right\} \, .\\
\]
In the same way we define the set $\mathcal{C}^{-(n+\alpha),D}$ and $\mathcal C^{-\frac\alpha2,-\alpha}$. The norms on these spaces are inherited by the classical dual spaces norm:
\begin{gather*}
\norm{\rho}_{-(n+\alpha)}=\sup\limits_{\norm{\phi}_{n+\alpha}\le 1}\langle\rho,\phi\rangle\,,\qquad\norm{\rho}_{-(n+\alpha),D}=\sup\limits_{\norm{\phi}_{n+\alpha,D}\le 1}\langle\rho,\phi\rangle\,,\\
\norm{\rho}\aam=\sup\limits_{\norm{\phi}\am\le1}\langle\rho,\phi\rangle \, .
\end{gather*}
The importance of these subspaces is given by the following result.
\begin{lemma}\label{prop:espen}
The space $\mathcal{C}^{-(n+\alpha)}$ is a norm-closed subset of $\left(\mathcal C^{n+\alpha}\right)'$. Moreover, if $\rho\in\mathcal C^{-n}$, then there exist a sequence $\{\rho_k\}_k\subset\mathcal C^{2+\alpha,D}$ such that $\rho_k\to\rho$ in $\mathcal C^{-(n+\alpha)}$.
\end{lemma}
\begin{proof}
Thanks to \cite[Lemma 2.6]{jakob}, we know that $\mathcal{C}^{-(n+\alpha)}$ is a norm-closed subset of $\left(\mathcal C^{n+\alpha}\right)'$. Let $\rho\in\mathcal C^{-n}$. Then there exist measures $\{\rho_\gamma\}_{|\gamma|\le n}$ such that
$
\rho=\sum\limits_{|\gamma|\le n}\partial^\gamma\rho_\gamma(dx)\,.
$ 
Defining $\rho_\gamma(\R^d\setminus\Omega)=0$, we can consider $\{\rho_\gamma\}_\gamma$ as measures on $\R^d$.

Let $\{n_\eps\}_\eps$ be smooth mollifiers. We define $\rho_\gamma^k:=\rho_\gamma*\eta_{\frac1k}$ and $\tilde\rho_k:=\sum\limits_{|\gamma|\le n}\partial^\gamma\rho^k_\gamma(dx)\,.$ Thanks to \cite[Lemma 2.3, Lemma 2.7]{jakob}, we have $\tilde\rho_k\in\mathcal C^\infty$ and $\tilde\rho_k\to\rho$ in $\mathcal C^{-(n+\beta)}$ for all $\beta>0$.

Then, we consider a nonnegative smooth function $\xi_k(s)\in\mathcal C^\infty(\R)$ such that $\xi_k(s)=1$ for $|s|\ge\frac1k$ and $\xi_k(0)=0$. We define $\rho_k(x)=\tilde\rho_k(x)\,\xi_k(d(x))$. Then we have $\rho_k\in\mathcal C^{2+\alpha,D}$ and the convergence to $\rho$ in $\mathcal C^{-(n+\alpha)}$ is preserved.
\end{proof}

\subsection{Subprobability Measures and Generalized Wasserstein Measure}
\label{subp}

Let us start the present subsection with a {\it proper}
notion of distance for elements belonging  to $\mathcal{P}^{sub} (\Omega)$.
Since the usual Wasserstein distance $W_p(\mu, \nu)$ is defined only when the two measures $\mu, \nu$ have the same mass, in what follows we rely on the so-called {\it generalized Wasserstein distance} which allows computing distance between measures with different masses, see e.g. \cite{PiccoliRossi, PiccoliRossi2} for a complete description and \cite{FigalliGigli,Santambrogio} for the corresponding probabilistic interpretation based on unbalanced optimal transport.

According to \cite{PiccoliRossi2}, we also introduce the following definition, which is the one we will use throughout the article.

%
%
%
%
%


\begin{definition}
Let $m_1,m_2\in\mathcal{P}^{sub}(\Omega)$ be two Borel sub-probability measures on $\Omega$. We call the generalized Wasserstein distance between $m_1$ and $m_2$, and we write $\dw(m_1,m_2)$ the quantity
\begin{equation}\label{wass1}
\dw(m_1,m_2):=\sup\limits_{Lip(\phi)\le 1 \, , \, \norm{\phi}_{C^0} \leq 1} \, \into \phi(x)d(m_1-m_2)(x)\,,
\end{equation}
\end{definition}

We remark that $\dw(m_1,m_2)$ is equivalent to the well-known {\em flat metric} over the space of Radon measures with finite mass on $\Omega$.

Furthermore, we define a suitable derivation of $U$ with respect to the (sub)measure $m$ as it appears in the Master Equation \eqref{Master}. Namely, we have:

\begin{definition}\label{dmu}
Let $U:\mathcal{P}^{sub}(\Omega)\to\R$. We say that $U$ is of class $\mathcal{C}^1$ if there exists a continuous map $K:\mathcal{P}^{sub}(\Omega)\times\Omega\to\R$ such that, for all $m_1$, $m_2\in\mathcal{P}^{sub}(\Omega)$ we have
\begin{equation}\label{deu}
\lim\limits_{s\to0}\frac{U(m_1+s(m_2-m_1))-U(m_1)}{s}=\into K(m_1,x)(m_2(dx)-m_1(dx))\,.
\end{equation}
\end{definition}

We define $\displaystyle\dm{U}(m,x):=K(m,x)$.

\begin{remark}
Condition \eqref{deu} defines uniquely the derivative $K$. Conversely, in the classical setting, where $U:\mathcal P(\Omega)\to\R$, see e.g. \cite{card,Ricciardi}, the function $K$ is defined up to an additive constant, since, for $c\in\R$ and $m_1,m_2\in\mathcal P(\Omega)$, $\into c\,d(m_2-m_1)=0$. Hence, in those cases a normalization condition is needed to uniquely identify $K$.
\end{remark}
%
Let us note that, exploiting Eq. \eqref{deu}, we have the following equality for $m_1,\,m_2\in\mathcal{P}^{sub}(\Omega)$, 	$$
U(m_2)-U(m_1)=\int_0^1\into\dm{U}(m_1+s(m_2-m_1),x)(m_2(dx)-m_1(dx))ds\,.
$$
which turns out to be very useful in our computations.
Moreover, we note that, if $\dm{U}$ is $\mathcal{C}^1$ in the space variable, we have
$$
\vert U(m_2)-U(m_1)\vert \le \sup\limits_m\left(\norminf{D_x\dm{U}(m,\cdot)}+\norminf{\dm{U}(m,\cdot)}\right)\dw(m_1,m_2)\,,
$$
suggesting us define the \emph{intrinsic derivative} of $U$ with respect to $m$, which directly appears in the Master Equation \eqref{Master}.
\begin{definition}\label{Dmu}
Let $U:\mathcal{P}^{sub}(\Omega)\to\R$. If $U$ is of class $\mathcal{C}^1$ and $\dm{U}$ is of class $\mathcal{C}^1$ with respect to the last variable, we define the intrinsic derivative $D_mU:\mathcal{P}^{sub}(\Omega)\times\Omega\to\R^d$ as
$$
D_mU(m,x):=D_x\dm{U}(m,x)\,.
$$
\end{definition}

\subsection{Assumptions}

We require the following hypotheses throughout the paper:

\begin{hp}\label{ipotesi}
Let $0<\alpha<1$. Assume that
\begin{itemize}
\item [(i)] $\norm{a(\cdot)}_{1+\alpha}<\infty$ and $\exists\,\mu>\lambda>0$ s.t. $\forall\xi\in\mathbb{R}^d$ $$ \lambda\vert \xi\vert ^2\le\langle a(x)\xi,\xi\rangle\le\mu\vert \xi\vert ^2\,;$$
\item [(ii)]$H:[0,T]\times\Omega\times\R^d\to\R$, $G:\Omega\times\mathcal{P}^{sub}(\Omega)\to\R$ and $F:[0,T]\times\Omega\times\mathcal{P}^{sub}(\Omega)\to\R$ are smooth functions with $H$ locally Lipschitz with respect to the last variable;
\item [(iii)]$\exists\, C>0$ s.t.
$$
0< H_{pp}(t,x,p)\le C I_{d\times d}\spazio;
$$

\item[(iv)] $F$ is increasing in the last variable, i.e.
$$
\into \left(F(t,x,m)-F(t,x,m')\right) d(m-m')(x)\ge0\,;
$$
moreover
$$
\sup\limits_{m\in\mathcal{P}^{sub}(\Omega)}\left(\norm{F(\cdot,\cdot,m)}_{\frac\alpha 2,\alpha}+\norm{\frac{\delta F}{\delta m}(\cdot,\cdot,m,\cdot)}_{\frac\alpha 2,\alpha,2+\alpha}\right)+\mathrm{Lip}\left(\dm{F}\right)\le C_F\spazio,
$$
with
$$
\begin{array}{l}
\displaystyle \mathrm{Lip}  \left(\dm{F}\right) := \smallskip \\
\hspace{0.2cm}	\displaystyle\sup\limits_{m_1\neq m_2}\left(\norm{m_1-m_2}_{-(2+\alpha),D}^{-1}\norm{\dm{F}(\cdot,\cdot,m_1,\cdot)-\dm{F}(\cdot,\cdot,m_2,\cdot)}_{\frac\alpha 2,\alpha,2+\alpha}\right);
\end{array}
$$

\item [(v)]
$G$ satisfies the same estimates as $F$ with $\alpha$ and $1+\alpha$ replaced by $2+\alpha$, i.e.
$$
\sup\limits_{m\in\mathcal{P}^{sub}(\Omega)}\left(\norm{G(\cdot,m)}_{2+\alpha}+\norm{\frac{\delta G}{\delta m}(\cdot,m,\cdot)}_{2+\alpha,2+\alpha}\right)+\mathrm{Lip}\left(\dm{G}\right)\le C_G\spazio,
$$
with
$$
\begin{array}{l}
\displaystyle 
\mathrm{Lip}\left(\dm{G}\right):=
\smallskip \\
\hspace{0.5cm}	\displaystyle
\sup\limits_{m_1\neq m_2}\left(\norm{m_1-m_2}_{-(2+\alpha),D}^{-1}\norm{\dm{G}(\cdot,m_1,\cdot)-\dm{G}(\cdot,m_2,\cdot)}_{2+\alpha,2+\alpha}\right)\spazio;
\end{array}
$$

\item[(vi)]

Besides the classical monotonicity assumption, we ask $F$
and $G$ to satisfy the following

$$
\Bigg\langle \Bigg\langle \frac{\delta F(t, \cdot,  m, \cdot)}{\delta m}, \rho \Bigg\rangle_y, \rho \Bigg\rangle_x \geq 0, \quad  
\rho \in \mathcal{C}^{-(2+\alpha),D}(\Omega) , 
$$

for all $m\in\mathcal P^{sub}(\Omega)$ and for all $t \in [0,T]$.

\item [(vii)] Moreover, we require the following Dirichlet boundary conditions:
\begin{align*}
\dm{F}(t,x,m,y)_{\vert y\in\partial\Omega}=0\,,\qquad \dm{G}(x,m,y)_{\vert y\in\partial\Omega}=0\,,\qquad
G(x,m)_{\vert x\in\partial\Omega}=0\,,
\end{align*}
for all $m\in\mathcal{P}^{sub}(\Omega)$.
\end{itemize}
\end{hp}
We note that in hypotheses $(vii)$ we have two standard compatibility assumptions for $G$ and $\dm{G}$, naturally linked to the boundary conditions framework. In particular, the condition on $G$ is essential to have a classical solution for both the Master Equation \eqref{Master} and the MFG system \eqref{eq:mfg}, whereas the condition on $\dm{G}$ is a compatibility condition for the linearized MFG system, see e.g. \emph{Corollary} \ref{delarue}.\\

We underline that condition on $\dm F$ seems unusual at a first sight. Technically speaking, this condition plays a crucial role to prove the analogue Dirichlet boundary condition of $\dm U$, and it is also mandatory to prove the well-posedness of the linearized MFG system in Proposition \ref{linearD}.

The conditions on $\dm F$ and $\dm G$ have also a meaning in terms of the $N$-players game. Roughly speaking, and arguing as in the introduction, assume that $F$ and $G$ are smooth functions, extended in the set $[0,T]\times\overline\Omega\times\mathcal P(\overline\Omega)$. The boundary conditions or $\dm F$ and $\dm G$ imply that
$$
F(t,x,m)=F(t,x,\tilde m)\quad\mbox{and}\quad  G(t,x,m)=G(t,x,\tilde m)\qquad\mbox{if}\quad m(\Omega)=\tilde m(\Omega).
$$
Hence, the cost functions for a player do not change if the players at the boundary change their position, which is totally natural in our model. Actually, the agents at the boundary play no role for the value function of the other players, and their specific position at the boundary must be unrelevant for the agents which are still playing their game.

\section{Well-posedness of the Master Equation}
\label{sec3}
We recall the definition of the function $U$, which turns out to be the solution of \eqref{Master}. Let $(t_0,m_0)\in[0,T]\times\mathcal P^{sub}(\Omega)$, and consider the solution $(u,m)$ of the MFG system with initial condition $m(t_0)=m_0$, i.e.
\begin{equation}
\label{eq:mfg}
\begin{cases}
\displaystyle - u_t - \mathrm{tr} (a(x) D^2 u) + H (t,x, Du) = F(t,x,m(t)) \, ,  \\
\displaystyle m_t - \nobt{m} - \mathrm{div} ( m H_p (t,x, Du)) = 0 \, , \\
m (t_0) = m_0, \hspace{1.8cm} u(x,T) = G(x, m(T)) \, , \\
\displaystyle \bdone{u}\,,\qquad\bdone{m}\,,
\end{cases}
\end{equation}
where, as always, a backward Hamilton-Jacobi-Bellman equation for the value function $u$ of the generic player is coupled with a forward Fokker-Planck equation for the density $m$ of the population. Then we define $U:[0,T]\times\Omega\times\mathcal P^{sub}(\Omega)\to\R$ as
\begin{equation}\label{eq:u}
U(t_0,x,m_0)=u(t_0,x) \, ,
\end{equation}

and we want to prove that, under certain assumptions, $U$ is the solution of the Master Equation \eqref{Master}.

A formal proof of existence and uniqueness of solutions for \eqref{Master} can be easily given directly using the very definition of $U$. The difficulty here relies on the proof of the $\mathcal{C}^1$ character of $U$ with respect to $m$. First, we establish the required regularity for $U$ and a related compatibility condition based on the boundary conditions. The existence and uniqueness theorem are rigorously proved, as we see in detail in the following theorem.

\begin{theorem}\label{settepuntouno}
Suppose hypotheses \ref{ipotesi} are satisfied. Take $U$ defined as in \eqref{eq:u}, suppose that $\dm{U}$ exists, being also bounded in $\mathcal{C}^{2+\alpha}_x\times C^{2+\alpha}_y$, uniformly in $t$ and continuously in $m$, and that boundary conditions of \eqref{Master} hold for $U$. Then $U$ is the unique classical solution $U$ of the Master Equation \eqref{Master}.

\begin{proof}
\emph{Existence}. By a density argument, we only need to prove that $U$ solves \eqref{Master} at $(t_0,x,m_0$), where $m_0$ is a smooth and positive function satisfying $\bdone{m_0}$ and $\into m_0(x)dx\le 1$.
Taking $(u,m)$ as the solution of  the $MFG$ system \eqref{eq:mfg} starting from $m_0$ at time $t_0$, by the definition of $U$, we have that for all $s>t_0$
$$
U(s,x,m(s))=u(s,x)\,,
$$
therefore, we can write
$$
\partial_t U(t_0,x,m_0)=\lim\limits_{h\to0} \frac{U(t_0+h,x,m_0)-U(t_0+h,x,m(t_0+h))}{h}+u_t(t_0,x)\,.
$$
Again by the very definition of $U$:
\begin{align*}
u_t(t_0,x)=-\mathrm{tr}(a(x)D^2_xU(t_0,x,m_0))+H(t_0,x,D_xU(t_0,x,m_0))-F(t_0,x,m_0)\,.
\end{align*}
As regards the limit part, we define $m_s:=(1-s)m(t_0)+sm(t_0+h)$. Since $U$ is $\mathcal{C}^1$   with respect to $m$, then:
\begin{align*}
& -\lim\limits_{h\to 0}\int_0^1\into\dm{U}(t_0+h,x,m_s,y)\frac{(m(t_0+h,y)-m(t_0,y))}h\,dy ds = \\
& 	\qquad - \into\dm{U}(t_0,x,m_0,y) \Bigg(\nobt{m(t_0,y)} \\ &
\hspace{4cm}  +\mathrm{div}\big(m(t_0,y)H_p(t_0,y,Du(t_0,y))\big)\Bigg)dy\,.
\end{align*}
Integrating by parts, and using the boundary conditions of $\dm{U}$ and $m$, we can rewrite the right-hand side as
\begin{align*}
&	\into\Bigg[H_p(t_0,y,D_x U(t_0,y,m_0))D_mU(t_0,x,m_0,y) \\ & \hspace{6cm}-\mathrm{tr}\Big(a(y)D_y D_mU(t_0,x,m_0,y)\Big)\Bigg]dm_0(y)\,,
\end{align*}
therefore
\begin{align*}
\displaystyle &\partial_t U(t,x,m)=-\mathrm{tr}\left(a(x)D_x^2 U(t,x,m)\right)+H\left(t,x,D_x U(t,x,m)\right)\\
\displaystyle & \quad -{\into}\mathrm{tr}\left(a(y)D_y D_m U(t,x,m,y)\right)dm(y)\\
\displaystyle & \quad + {\into} D_m U(t,x,m,y)\cdot H_p(t,y,D_x U(t,y,m))dm(y)- F(t,x,m)\,,
\end{align*}
hence concluding the existence part.

\emph{Uniqueness}. Suppose that there exists a second solution $V$ of \eqref{Master}. We fix $(t_0,m_0)$, with $m_0$ smooth s.t. $\bdone{m_0}$, and we take  $\tilde{m}$ as the solution of
\begin{equation*}
\begin{cases}
\tilde{m}_t-\nobt{\tilde{m}}-\mathrm{div}\big(\tilde{m}H_p(t,x,D_xV(t,x,\tilde{m}))\big)=0\,,\\
\tilde{m}(t_0)=m_0\,,\\
\bdone{\tilde{m}}\,
\end{cases}
\end{equation*}
which is well-posed since $D_xV$ is Lipschitz continuous with respect to the measure. Then we define $\tilde{u}(t,x)=V(t,x,\tilde{m}(t))$ and we want to prove that $\tilde{u}$ solves a Hamilton-Jacobi equation. Using the equations of $V$ and $\tilde{m}$, we get
\begin{align*}
\tilde{u}_t(t,x)=&\,V_t(t,x,\tilde{m}(t))+\into\dm{V}(t,x,\tilde{m}(t),y)\,\tilde{m}_t(t,y)\,dy\\
=&\,V_t(t,x,\tilde{m}(t))+\into\dm{V}(t,x,\tilde{m}(t),y)\,\nobt{\tilde{m}(t,y)}\,dy\\+&\,\into \dm{V}(t,x,\tilde{m}(t),y)\,\mathrm{div}\big(\tilde{m}H_p(t,x,D_xV(t,x,\tilde{m}))\big)\,dy\,.
\end{align*}
Previous integrals can be easily estimated with an integration by parts, and taking into account the boundary condition for $V$. As regards the first, we have
\begin{align*}
&\into\dm{V}(t,x,\tilde{m}(t),y)\,\nobt{\tilde{m}(t,y)}\,dy=\smallskip  \\
\displaystyle & \qquad \into\mathrm{tr}(a(y)D_yD_mV(t,x,\tilde{m}(t),y))\,\tilde{m}(t,y)\,dy\,,
\end{align*}

while for the second
\begin{align*}
\displaystyle	&\into \dm{V}(t,x,\tilde{m}(t),y)\,\mathrm{div}\big(\tilde{m}H_p(t,x,D_xV(t,x,\tilde{m}))\big)\,dy \smallskip \\
\displaystyle & \qquad =-\into H_p(t,x,D_xV(t,x,\tilde{m}))D_mV(t,x,\tilde{m},y)\tilde{m}(t,y)dy\,.
\end{align*}
Previous estimates allow us to write
\begin{align*}
\tilde{u}_t(t,x)=&\,V_t(t,x,\tilde{m}(t))+\into\mathrm{tr}(a(y)D_yD_mV(t,x,\tilde{m},y))\,d\tilde{m}(y) \smallskip \\ & \qquad -\,\into H_p(t,y,D_xV(t,y,\tilde{m})) D_mV(t,x,\tilde{m},y)\,d\tilde{m}(y)\smallskip \\=&\,-\mathrm{tr}(a(x)D^2\tilde{u}(t,x))+H(t,x,D\tilde{u}(t,x))-F(t,x,\tilde{m}(t))\,,
\smallskip			\end{align*}
hence $(\tilde{u},\tilde{m})$ is a solution of the MFG system \eqref{eq:mfg}. Then, by uniqueness of solutions for the MFG system, we get $(\tilde{u},\tilde{m})=(u,m)$, implying that $V(t_0,x,m_0)=U(t_0,x,m_0)$, whenever $m_0$ is smooth. Therefore, by density, the uniqueness is proved.
\end{proof}
\end{theorem}
In the next two sections, we will prove the $\mathcal{C}^1$ character of $U$ with respect to $m$, the $\mathcal{C}^2$ regularity in $x$ and $y$ of $\dm{U}$ and the boundary conditions for $U$, hence justify the hypotheses previously given, also allowing us to then apply Theorem \ref{settepuntouno} and finally prove the well-posedness of \eqref{Master}.

\section{The Fokker-Planck equation and the MFG system}
\label{sec4}

Let us first give the following technical lemma, providing a regularity result for a linear PDE with non-homogeneous drift.
\begin{lemma}\label{sonobravo}
Suppose $a$ satisfies hypothesis $\textit{(i)}$ of \ref{ipotesi} ,  let $q=\frac{d+2}{1-\alpha}$, $b,f\in L^q(Q_T)$, $\psi\in\mathcal{C}^{1+\alpha,D}(\Omega)$, and $z$ be the solution of 
\begin{equation*}
\begin{cases}
-z_t-\mathrm{tr}(a(x)D^2z)+b(t,x)\cdot Dz=f(t,x)\,,\\
z(T)=\psi\,,\\
z_{\vert \partial\Omega}=0\,.
\end{cases}
\end{equation*}
Then $z$ satisfies, for a certain $C>0$ depending on $\norm{a}_{L^q}$ and $\norm{b}_{L^q}$,
\begin{equation}\label{estensione}
\norm{z}\amu\le C\left(\norm{f}_{L^q}+\norm{\psi}_{1+\alpha}\right)\,.
\end{equation}
Moreover, if $b, f\in\mathcal{C}^{0,\alpha}(Q_T)$ and $\psi\in\mathcal{C}^{2+\alpha,D}(\Omega)$, it holds
\begin{equation}\label{espansione}
\norm{z}\amv\le C\left(\norm{f}_{0,\alpha}+\norm{\psi}_{2+\alpha}\right)\,.
\end{equation}
\begin{proof}
We exploit ideas stated in \cite{Ricciardi}. In particular, if $f\in\mathcal{C}(Q_T)$, $b\in\mathcal{C}(\Omega)$, then we conclude with \emph{Theorem 5.1.1} of \cite{lunardi}. While, in the general case, we write $z=z_1+z_2$, where $z_1$ and $z_2$ respectively satisfy: 

\begin{equation*}
\begin{cases}
-{(z_1)}_t-\mathrm{tr}(a(x)D^2z_1)=0\,, \smallskip \\
z_1(T)=\psi\,,\smallskip \\
{z_1}_{\vert \partial\Omega}=0\,, \smallskip
\end{cases}
\end{equation*}
and 		
\begin{equation*}
\begin{cases}
-{(z_2)}_t-\mathrm{tr}(a(x)D^2z_2)+b\cdot Dz_2=f-b\cdot Dz_1\,,\smallskip \\
z_2(T)=0\,, \smallskip \\
{z_2}_{\vert \partial\Omega}=0\,. \smallskip
\end{cases}
\end{equation*}

For $z_1$, we apply \emph{Theorem 5.1.11} of \cite{lunardi} to obtain $\norm{z_1}\amu\le C\norm{\psi}_{1+\alpha}\,.$ For $z_2$, by Corollary of \emph{Theorem IV.9.1} in \cite{lsu}, we have
$$
\norm{z_2}\amu\le C\norm{f-bDz_1}_{L^q}\le C\left(\norm{f}_{L^q}+\norm{\psi}_{1+\alpha}\right)\,,
$$
concluding the first part. Then, if $b,f\in\mathcal{C}^{0,\alpha}(Q_T)$ and $\psi\in\mathcal{C}^{2+\alpha,D}$, by \eqref{estensione}, we have $f-b\cdot Dz\in\mathcal{C}^{0,\alpha}$, allowing to apply \emph{Theorem 5.1.13} in \cite{lunardi} to conclude.
\end{proof}
\end{lemma}

\subsection{The Fokker-Planck equation}\label{subsection}

From now on, $q$ will denote the quantity $q:=\frac{d+2}{1-\alpha}$, and $p$ will be the conjugate exponent of $q$, i.e. $p:=\frac{d+2}{d+1+\alpha}\,.$ 

The present section is focused on the study of the following Fokker-Planck equation
\begin{equation}\label{linfp}
\begin{cases}
\mu_t-\mathrm{div}(a(x)D\mu)-\mathrm{div}(\mu b)=\mathrm{div}(c)\,,\\
\mu(t_0)=\mu_0\,,\\
\mu_{\vert \partial\Omega}=0\,,
\end{cases}
\end{equation}
where $c\in L^1(Q_T)$, $\mu_0\in\mathcal{C}^{-(1+\alpha)}$, and $b\in L^q(Q_T)$. 
Let us underline that the latter equation plays a fundamental role
in studying both the Mean Field Games system and the linearized one.

The main difficulty here lies in the low regularity of both data and coefficients. Hence, the main idea is to start with the regular case, where we have the existence and uniqueness of solutions,
to then obtain estimates that we will then exploit to pass to the limit in the general case.

Let us start by giving a suitable definition of a solution for the system \eqref{linfp}:
\begin{definition}\label{canzonenuova}
Let $c\in L^1$, $\mu_0\in\mathcal{C}^{-1}$, $b\in L^q$. We say that a function $\mu\in\mathcal{C}([0,T];\mathcal{C}^{-(1+\alpha),D})\cap L^1(Q_T)$ is a weak solution of \eqref{linfp} if, for all $t\in(t_0,T]$ and $\phi$ satisfying in $[t_0,t]\times\Omega$ the following linear equation
\begin{equation}\label{hjbfp}
\begin{cases}
-\phi_t-\mathrm{div}(aD\phi)+bD\phi=\psi\,,\\
\phi(t)=\xi\,,\\
\phi_{\vert \partial\Omega}=0\,,
\end{cases}
\end{equation}
with $\psi\in L^\infty(\Omega)$ and $\xi\in\mathcal{C}^{1+\alpha,D}$, it holds true:
\begin{equation}\label{weakmu}
\langle \mu(t),\xi\rangle+\int_{t_0}^t\into\mu(s,x)\psi(s,x)\,dxds=\langle\mu_0,\phi(t_0,\cdot)\rangle-\int_{t_0}^t\into c(s,x)\cdot D\phi(s,x)\,dxds\,,
\end{equation}
where $\langle \cdot,\cdot\rangle$ denotes, respectively, the duality between $\mathcal{C}^{-(1+\alpha),D}$ and $\mathcal{C}^{1+\alpha,D}$, $\mathcal{C}^{-(1+\alpha)}$ and $\mathcal{C}^{1+\alpha}$.
\end{definition}

Note that Lemma \ref{sonobravo} implies $\phi \in \mathcal{C}^{\frac{1+\alpha}2,1+\alpha}$, therefore the terms on the right-hand side are well-defined and the definition is well-posed.

\begin{proposition}\label{peggiodellagerma}
Let $c\in L^1$, $\mu_0\in\mathcal{C}^{-1}$, $b\in L^q$. Then there exists a unique solution for \eqref{linfp}, which satisfies, for a certain $C$ depending on $\norm{a}_{L^q}$ and $\norm{b}_{L^q}$,




\begin{equation}\label{stimefokker}
\supt\norm{\mu(t)}\amc+\norm{\mu}_{L^p}\le C\left(\norm{\mu_0}_{-(1+\alpha)}+\norm{c}_{L^1}\right)\,,
\end{equation}

Moreover, if $\mu^n_0\to\mu_0$ in $\mathcal{C}^{-(1+\alpha)}$, $b^n\to b$ in $L^q$, $c^n\to c$ in $L^1$,  we have $\mu^n\to\mu$ in $\mathcal{C}([0,T];\mathcal{C}^{-(1+\alpha),D})\cap L^p(Q_T)$, where $\mu^n$ and $\mu$ are the solutions related  to $(\mu^n_0,b^n,f^n)$ and $(\mu_0,b,f)$.

\end{proposition}

The choice of the space $\mathcal C^{-1}$ motivated by the initial datum in Eq. \ref{eq:mfg} enables a smooth approximation scheme within the framework of dual Hölder spaces, while maintaining uniform estimates in the stronger  $\mathcal{C}^{-(1+\alpha),D}$ norm.

\begin{proof}
Without loss of generality, we can consider $t_0=0$.

\emph{Existence: smooth case.} Assume that $f$, $b$ and $\mu_0$ are smooth functions, with $\mu_0$ satisfying  $\mu_0(x)_{\vert x\in\partial\Omega}=0\,.$ Then, splitting the divergence terms in \eqref{linfp}, we obtain a linear equation and the existence of solutions is guaranteed by \cite{lsu,lunardi}.

Let $\phi$ be a solution of \eqref{hjbfp}, with $\psi=0$ and $\xi\in\mathcal{C}^{1+\alpha,D}$. Using $\phi$ as a test function for $\mu$, we get  by Lemma \ref{sonobravo}
\begin{align*}\label{rhs}
&   \displaystyle \langle \mu(t),\xi\rangle=\langle\mu_0,\phi(0,\cdot)\rangle-\int_0^t\into c(s,x)\cdot D\phi(s,x)\,dxds \smallskip \\ & \hspace{1.2cm} \displaystyle \le C\norm{\xi}_{1+\alpha}\left(\norm{\mu_0}_{-(1+\alpha)}+
\norm{c}_{L^1}\right)\,.
\end{align*}

Passing to the $sup$ for $\norm{\xi}_{1+\alpha}\le 1$, we obtain
\begin{equation}
\supo\norm{\mu(t)}\amc\le C\left(\norm{\mu_0}_{-(1+\alpha)}+\norm{c}\amf\right)\,,
\end{equation}
and we are left to prove the $L^p$ bound for $\mu$. To this end, we take $\phi$ as the solution of \eqref{hjbfp}, with $t=T$, $\xi=0$ and $\psi\in L^q\,.$ Then, again by Lemma \ref{sonobravo}, we have
\begin{align*}
&\displaystyle \intif\mu\psi\,dxds=\langle\mu_0,\phi(0,\cdot)\rangle-\int_0^T\into c(s,x)\cdot D\phi(s,x)\,dxds
\smallskip \\ & \hspace{2cm} \displaystyle \le C\norm{\psi}_{L^q}\left( \norm{\mu_0}_{-(1+\alpha)}+\norm{c}\amf \right)\,.
\end{align*}
Passing to the $\sup$ for $\norm\mu_{L^q}\le1$, we get
$$
\norm{\mu}_{L^p}\le C\left( \norm{\mu_0}_{-(1+\alpha)}+\norm{c}\amf \right)\,,
$$
which concludes the proof in the regular case.

\emph{Existence: general case.} We proceed adapting and adjusting \cite[Proposition 5.3]{Ricciardi}.
By Lemma \ref{prop:espen}, we take smooth data $\mu_{0}^k$, $c_k$, $b_k$, with $\mu_0^k(x)=0$ for $x\in\partial\Omega$, and converging to $\mu_{0}$, $c$, $b$, respectively in $\mathcal{C}^{-(1+\alpha),D}$, $L^1$ and $L^q$.


We consider $\mu^k$ as the solution of \eqref{linfp}. The previous convergences imply, for a certain $C>0$,
\begin{equation*}	\begin{array}{c}
\vert\vert {\mu_{0}^k}\vert\vert \amb\le C\norm{\mu_0}\amb\, , \smallskip \\ \norm{b_k}_{L^q}\le C\norm{b}_{L^q}\, , \smallskip \\ \vert\vert {c_k}\vert\vert_{L^1}\le C\norm{c}_{L^1}\,. \smallskip
\end{array}
\end{equation*}
Considering the function $\mu^{k,h}:=\mu^k-\mu^h$, which satisfies \eqref{linfp} with data $b=b_k$,\\$c=c_k-c_h+\mu_h(b_k-b_h)$, $\mu^0=\mu_{0}^k-\mu_{0}^h$, then, by estimates \eqref{stimefokker},
we have
\begin{align*}
&\sup_t \vert\vert {\mu^{k,h}(t)} \vert\vert \amc\,+\, \vert\vert {\mu^{k,h}} \vert\vert _{L^p}\\\le C&\left( \vert\vert {\mu_{0}^k-\mu_{0}^h} \vert\vert _{-(1+\alpha)}+ \vert\vert {c_k-c_h} \vert\vert _{L^1}+ \vert\vert {\mu_h(b_k-b_h)} \vert\vert_{L^1}\right)\,,
\end{align*}

By the uniform bound of $\{\mu_k\}_k$ in $L^p$, 
the last term can be estimated as follows 		\begin{align}\label{luigicoibluejeans}
\norm{\mu^h(b_k-b_h)}_{L^1}\le C\vert\vert {b_k-b_h}\vert\vert _{L^q}\,,
\end{align}
and since $\{\mu_{0}^k\}_k$, $\{c_k\}_k$ and $\{b_k\}_k$ are Cauchy sequences, the right-hand side goes to $0$ and $\{\mu^k\}_k$ is a Cauchy sequence, too. Then $\mu^k\to\mu$ in $\mathcal{C}([0,T];\mathcal{C}^{-(1+\alpha),N})\cap L^p(Q_T)$, for a certain $\mu$ also satisfying \eqref{stimefokker}.
We are then left to prove that $\mu$ satisfies \eqref{weakmu}, for $\phi$ satisfying \eqref{hjbfp}. To this end, we consider the test function $\phi_k$ related to $\mu^k$, so that
\begin{equation}\label{muoviti}
\langle \mu^k(t),\xi\rangle+\inti\mu^k(s,x)\psi(s,x)\,dxds=\langle\mu_{0}^k,\phi_k(0,\cdot)\rangle-\int_0^t\into c_k(s,x)\cdot D\phi_k(s,x)\,dxds\,,
\end{equation}
Concerning convergence of $\phi_k$ towards $\phi$, since $\phi_k-\phi$ satisfies
\begin{equation*}
\begin{cases}
-(\phi_k-\phi)_t-\mathrm{div}(aD(\phi_k-\phi))+b_kD(\phi_k-\phi)=(b_k-b)D\phi\,,\\
(\phi_k-\phi)(t)=0\,,\\
(\phi_k-\phi)_{\vert \partial\Omega}=0\,,
\end{cases}
\end{equation*}
by Lemma \ref{sonobravo}, we have
$$
\vert\vert {\phi_k-\phi}\vert\vert \amu\le C\vert\vert {(b_k-b)D\phi}\vert\vert _{L^q}\to0\,,
$$
allowing us to pass to the limit in \eqref{muoviti}, hence ending the existence part.

\emph{Uniqueness and stability.} Let $\mu_1$ and $\mu_2$ be two solutions. Then $\mu:=\mu_1-\mu_2$ solves the same problem with $c=\mu_0=0$, and by \eqref{stimefokker} we have $\norm{\mu}_{L^1}=\supo\norm{\mu(t)}_{-(1+\alpha),D}=0$,  concluding the uniqueness part.

For the stability part, we consider $c_n\to c$, $\mu_{0}^n\to\mu_0$ and $b_n\to b$. Then the function $\mu^n-\mu$ satisfies \eqref{linfp}, with $b\,,\mu_0$ and $f$ respectively replaced by $b^n$, $\mu^n_0-\mu_0$, and $c^n-c+\mu(b^n-b)$. Hence, \eqref{stimefokker} implies
$$
\begin{array}{c}		
\supo\norm{\mu^n-\mu}\amc\,+\norm{\mu^n-\mu}_{L^p}\\ \le C\left(\vert\vert {\mu^n_0-\mu_0} \vert\vert _{-(1+\alpha)}+ \vert\vert {c^n-c} \vert\vert _{L^1}+ \vert\vert {\mu(b^n-b)} \vert\vert_{L^1}\right)\,.
\end{array}
$$
Proceeding as done for the existence in the general case, we derive the right-hand side convergence to $0$, which implies $\mu^n\to\mu$ in $\mathcal{C}([0,T];\mathcal{C}^{-(1+\alpha),D})\cap L^p(Q_T)$, hence concluding the proof.
\end{proof}	

As a direct consequence of Prop. \eqref{peggiodellagerma},
we can derive a further estimate in the space $\mathcal C^{-(2+\alpha),D}$.

\begin{corollary}
\label{corollary}
Suppose hypotheses of Proposition \ref{peggiodellagerma} are satisfied together with $b\in\mathcal{C}^{\frac\alpha 2,\alpha}$. Then the function $\mu$ satisfies
\begin{align}
\supo\norm{\mu(t)}_{-(2+\alpha),D}+\norm{\mu}\aam\le C\left(\norm{\mu_0}_{-(2+\alpha)}+\norm{c}\amf\right)\,,\label{forsemisalvo}
\end{align}
\end{corollary}
\begin{proof}
Since $b\in\mathcal{C}^{\frac\alpha 2,\alpha}$, we can take $\phi$ as the solution of \eqref{hjbfp}, with $\xi\in C^{2+\alpha,D}(\Omega)$, $\psi=0$, so that, applying the results of \cite{lsu}, \cite{lunardi}, we have $\norm{\phi}\amd\le C\norm{\xi}_{2+\alpha},$ and \eqref{weakmu} implies
\begin{align*}
&    \displaystyle   \langle\mu(t),\xi\rangle=\langle\mu_0,\phi(0,\cdot)\rangle-\int_0^T\into c(s,x)\cdot D\phi(s,x)\,dxds \smallskip \\
\displaystyle & \hspace{1.3cm} \le C\left(\norm{\mu_0}_{-(2+\alpha)}+\norm{c}\amf\right)\norm{\xi}_{2+\alpha}\,.
\end{align*}

Now, as in Proposition \ref{peggiodellagerma}, we choose $\phi$ as the solution of \eqref{hjbfp}, with $t=T$, $\xi=0$ and $\psi\in\mathcal C^{\frac\alpha2,\alpha}\,.$ Then, again by Lemma \ref{sonobravo}, we have
\begin{align*}
&\displaystyle \intif\mu\psi\,dxds=\langle\mu_0,\phi(0,\cdot)\rangle-\int_0^T\into c(s,x)\cdot D\phi(s,x)\,dxds
\smallskip \\ & \hspace{2cm} \displaystyle \le C\norm{\psi}\am\left( \norm{\mu_0}_{-(2+\alpha)}+\norm{c}\amf \right)\,.
\end{align*}
Passing to the $\sup$ in the two inequalities for $\norm\psi\am\le1$ and $\norm{\xi}_{2+\alpha}\le 1$, we get \eqref{forsemisalvo}.
\end{proof}

\begin{remark}\label{rem:espen}
Observe that the analogue of this Corollary in the Neumann case is given by \cite[Corollary 5.4]{Ricciardi}. However, in order to adjust the issues raised in \cite{jakob}, we have refined the estimate \eqref{forsemisalvo}, by proving a bound for $\mu\in\mathcal C^{-\left(\frac\alpha2,\alpha\right)}$. This estimate is crucial to prove stronger regularity results for the intrinsic derivative of $U$ in the measure variable, $D_mU$. We stress the fact that similar adjustments can also be done in the Neumann case \cite{Ricciardi}.
\end{remark}

\subsection{Mean Field Games system and Lipschitz regularity of U}
Now, we turn to the study of the MGF system \eqref{eq:mfg} and its properties to derive essential results that we will then exploit to obtain some first regularity estimates of the solution  $U$ of the Master Equation.

In particular, the MFG system has the following form:
\begin{equation}\label{meanfieldgames}
\begin{cases}
-\partial_t u - \mathrm{tr}(a(x)D^2u) +H(t,x,Du)=F(t,x,m(t))\,,\\
\partial_t m - \sum\limits_{i,j} \partial_{ij}^2 (a_{ij}(x)m) -\mathrm{div}(mH_p(t,x,Du))=0\,,\\
u(T)=G(x,m(T))\,,\qquad m(0)=m_0\,,\\
\diric{u}\,,\qquad\diric{m}\,.
\end{cases}
\end{equation}

The first result is obtained by studying some regularity properties of \eqref{meanfieldgames} uniformly in $m_0$.

\begin{proposition}
\label{estimates}
The system \eqref{meanfieldgames} has a unique classical solution $(u,m)$, with $u\in \mathcal{C}^{1+\frac{\alpha}{2},2+\alpha}$ and $m\in\mathcal{C}([t_0,T];\mathcal{P}^{sub}(\Omega))$. Moreover, $m(t)$ has a positive density for $t>t_0$, and the following estimates hold:
\begin{equation}\label{first}
\norm{m}\amm{\frac{\alpha}{2}}+\norm{m}_{L^p}+\norm{u}_{1+\frac{\alpha}{2},2+\alpha}\le C\,.
\end{equation}
Furthermore, if $(u_1,m_1)$ and $(u_2,m_2)$ are two solutions of \eqref{meanfieldgames}, with $m_1(t_0)=m_{01}$, $m_2(t_0)=m_{02}$, then it holds 
\begin{equation}\begin{split}\label{lipsch}
\norm{u_1-u_2}\amv+\norm{m_1-m_2}_{L^p(Q_T)}&\le C\dw(m_{01},m_{02})\,,\\
\supo\norm{m_1(t)-m_2(t)}\amc& \le C\dw(m_{01},m_{02})\,.
\end{split}\end{equation}
\begin{proof}
\emph{Step 1: Existence and uniqueness of solutions.} We want to apply Schauder's fixed-point Theorem. Let $X$ be the convex set, closed for the uniform distance, defined as follows:
\begin{equation}
X:=
\left\{\gamma\in\mathcal{C}([t_0,T];\mathcal P^{sub}(\Omega))\mbox{ s.t. }\norm{\gamma}\amm{\frac\alpha 2}\le M\,\,\, \forall s,t\in[t_0,T]
\right\}\,,
\end{equation}
where $M$ will be specified later. We define a map $\Phi:X\to X$ as follows.\\
For $\gamma\in X$, let $u=u_\gamma$ be the solution of the HJB equation:		\begin{equation}\label{hj}
\begin{cases}
-u_t-\mathrm{tr}(a(x)D^2u)+H(t,x,Du)=F(t,x,\gamma(t))\,, \smallskip \\
u(T)=G(x,\gamma(T))\,, \smallskip \\
u_{\vert \partial\Omega}=0\,.\smallskip
\end{cases}
\end{equation}
Using hypotheses on $F$ and $G$ and \textit{Theorem V.6.1} of \cite{lsu}, we know that there exists a unique classical solution $u\in\mathcal C^{1+\frac\alpha 2,\alpha}$. Moreover, by Taylor's formula, we can write $H(t,x,Du)=H(t,x,0)+V(t,x)\cdot Du$ for a certain $V\in L^\infty$, thanks to the global boundedness of $Du$,  obtaining a linear equation satisfied by $u$.
Then, exploiting both the Corollary of \emph{Theorem IV.9.1} and \emph{Theorem IV.5.2} of \cite{lsu} we get
\begin{align*}
\norm{u}_{1+\frac{\alpha}{2},2+\alpha}\le C\left(\norm{F(\cdot,\cdot,\gamma(\cdot))}_{\frac{\alpha}{2},\alpha}+\norm{G(\cdot,\gamma(T))}_{2+\alpha}\right)\,,
\end{align*}
Define $\Phi(\gamma)=m$, where $m\in\mathcal{C}([t_0,T];\mathcal{P}^{sub}(\Omega))$ is the solution of the FP equation
\begin{equation}\label{fpk}
\begin{cases}
m_t-\sum\limits_{i,j} \partial_{ij}^2 (a_{ij}(x)m) -\mathrm{div}(mH_p(t,x,Du))=0\,, \smallskip \\
m(t_0)=m_0\,, \smallskip \\
m_{\vert \partial\Omega}=0\,. \smallskip
\end{cases}
\end{equation}
The existence of a unique solution for \eqref{fpk} is guaranteed by Proposition \ref{peggiodellagerma}. Moreover, we know from \eqref{stimefokker} that
$$
\supt\norm{m(t)}\amc+\norm{m}_{L^p}\le C\norm{m_0}_{-(1+\alpha)}\,,
$$
for a certain $p>1$.
Therefore, to check that $m\in X$, we are left with proving that, for some $C>0$, it holds
$$
\norm{m(t)-m(s)}\amc\le C\vert t-s\vert ^{\frac\alpha 2}\,,\qquad\mbox{for all }t\neq s\,.
$$
Subtracting the distributional formulation \eqref{weakmu} in $t$ and $s$, we have
\begin{equation}\begin{split}\label{sotis}
\into \xi(x)m(t,dx)-\into\phi(s,x)m(s,dx)+\int_s^t\into\psi(r,x)m(r,dx)dr=0\,,
\end{split}\end{equation}
for each $\xi\in\mathcal{C}^{1+\alpha,D}$, $\psi\in L^\infty$ and $\phi$ satisfying
\begin{equation}\label{coglia}
\begin{cases}
-\phi_t-\mathrm{tr}(a(x)D^2\phi)+H_p(t,x,Du)\cdot D\phi=\psi\,, \smallskip \\
\phi(t)=\xi\,, \smallskip \\
\phi_{\vert \partial\Omega}=0\,. \smallskip
\end{cases}
\end{equation}
We choose $\psi=0$. Thanks to Lemma \ref{sonobravo}, we know that $\phi\in\mathcal{C}^{1+\frac\alpha 2,1+\alpha}$ and its norm is bounded according to \eqref{estensione}. Coming back to \eqref{sotis}, we obtain

\begin{equation}
\begin{array}{c}
\displaystyle \into\xi(x)(m  (t,dx)-m(s,dx)) =\into(\phi(s,x)-\phi(t,x))m(s,dx) \medskip \\ 
\displaystyle \qquad \qquad \le C\norm{\phi(t)-\phi(s)}_{1+\alpha}
\norm{m(s)}\amc \le C\vert t-s\vert ^{\frac\alpha 2}\norm{m_0}_{-(1+\alpha)}\norm{\xi}_{1+\alpha}\,, \smallskip
\end{array}
\end{equation}
and taking the $\sup$ over the

$\xi\in\mathcal{C}^{1+\alpha,D}$ with $\norm{\xi}_{1+\alpha}\le 1$,
\begin{equation}
\norm{m(t)-m(s)}\amc\le C\vert t-s\vert ^{\frac\alpha 2}\norm{m_0}_{-(1+\alpha)}\,.
\end{equation}


Choosing $M=C\norm{m_0}_{-(1+\alpha)}$, we have proved that $m\in X$.\\


Since $X$ is convex and closed and $\Phi(X)\subseteq X$, to apply Schauder's theorem, we need to show that:
\begin{itemize}
\item $\Phi(X)$ is relatively compact;
\item $\Phi$ is continuous.
\end{itemize}

In this way, the closure of the convex envelope of $\Phi(X)$, say $\hat{X}:=\overline{\mathrm{inv}(\Phi(X))}$, is compact and convex, and for the closure and the convexity of $X$ we have $\Phi(X)\subseteq \hat{X}\subseteq X$. So we can consider the restriction $\Phi_\vert  : \hat{X}\to \hat{X}\,,$ which satisfies the classical hypotheses of the Schauder's Theorem. Hence, the existence of a fixed point remains guaranteed.

We start proving the relatively compactness of $\Phi(X)$. Let ${\{\gamma_n\}}_n\subset X$ and let $u_n$ and $m_n$ be the related solutions. Applying Ascoli-Arzel\`a's Theorem we have  $u_{n_k}\to u$ in $\mathcal{C}^{1,2}$,
for a certain subsequence $\{u_{n_k}\}_k$ and $u\in X$.\\
To prove the convergence of $\{m_{n_k}\}_n$, we take $\phi_{n_k}$ as the solution of \eqref{coglia} with $Du$ replaced by $Du_{n_k}$ and $\psi=0$. The difference $\phi_{k,h}:=\phi_{n_k}-\phi_{n_h}$ satisfies
\begin{equation}
\begin{cases}
\displaystyle -(\phi_{k,h})_t-\mathrm{tr}(a(x)D^2\phi_{k,h})+H_p(t,x,Du_{n_k})\!\cdot\! D\phi_{k,h} \smallskip \\ \hspace{3cm} = (H_p(t,x,Du_{n_h})-H_p(t,x,Du_{n_k}))\!\cdot\! D\phi_{n_h}\,, \smallskip \\
\phi_{k,h}(t)=0\,, \smallskip \\
\bdone{\phi_{k,h}}\,, \smallskip
\end{cases}
\end{equation}
then Lemma \ref{sonobravo} implies
\begin{equation}
\begin{array}{c}
\displaystyle \norm{\phi_{k,h}}\amu \le C\norminf{(H_p(t,x,Du_{n_h})-H_p(t,x,Du_{n_k}))\cdot D\phi_{n_h}}\\ \displaystyle \le C\norminf{Du_{n_h}-Du_{n_k}}\le \omega(k,h)\,,
\end{array}
\end{equation}
where $\omega(k,h)\to 0$ when $k,h\to\infty$.

Using \eqref{sotis} with $(m_{n_k},\phi_{n_k})$ and $(m_{n_h},\phi_{n_h})$, for $k,h\in\mathbb{N}$, $s=t_0$, subtracting the two equalities, we get
\begin{align*}
\supt\norm{m_{n_k}(t)-m_{n_h}(t)}\amc\le\omega(k,h)\,\implies \exists\, m\in X \mbox{ s.t. }m_{n_k}\to m \mbox{ in }X\,,
\end{align*}
hence concluding the compactness part. The continuity is an easy consequence of the previous arguments.

In particular, we apply Schauder's theorem and obtain a classical solution of the problem \eqref{meanfieldgames}. The estimate \eqref{first} follows from the above estimates for \eqref{hj} and \eqref{fpk}.

We skip the uniqueness part, which is a standard argument, see, e.g., \emph{Proposition 3.3} of \cite{Ricciardi}.

\emph{Step 2.} Let $(u_1,m_1)$ and $(u_2,m_2)$ be two classical solutions of \eqref{meanfieldgames}, with $m_1(t_0)=m_{01}$, $m_2(t_0)=m_{02}$.  We take $\phi$ as the solution of \eqref{coglia} related to $u_1$, with $\psi=0$, and we note that $\phi$ is also a good test function for the equation of $m_2$, since it satisfies, for  $\psi=\left(H_p(t,x,Du_2)-H_p(t,x,Du_1)\right)\cdot D\phi\in L^\infty$,
\begin{equation}
\begin{cases}
-\phi_t-\mathrm{tr}(a(x)D^2\phi)+H_p(t,x,Du_2)\cdot D\phi=\psi\,, \smallskip \\
\phi(t)=\xi\,, \smallskip \\
\phi_{\vert \partial\Omega}=0\, . \smallskip
\end{cases}
\end{equation}
Then, subtracting the weak formulations of $m_1$ and $m_2$ related to $\psi$, we find
\begin{align}
\label{immigrato}
&	\displaystyle \into\xi(x)(m_1(t)-m_2(t))\,dx =  \smallskip \\ & \displaystyle \hspace{.2cm} \int_{t_0}^t\into(H_p(t,x,Du_1)-H_p(t,x,Du_2))D\phi\, m_2(s,x)dx\,ds +\langle\phi(0,\cdot),m_{01}-m_{02}\rangle\,.
\end{align}
By Lipschitz continuity of both $\phi$ and $H_p$, with respect to $p$, we get
\begin{align}
\into \xi(x)(m_1(t)-m_2(t))\,dx\le C\int_{t_0}^t\into \vert Du_1-Du_2\vert  m_2(s,x)dx\,ds+C\dw(m_{01},m_{02})\,,
\end{align}
We want to estimate the first term in the right-hand side with Young's inequality. To this end, we consider the quantities
$$
\tilde{c}=\int_{t_0}^t\into m_2(s,x)\,dxds\,,\qquad \tilde{m}_2(s,x)=\frac{m_2(s,x)}{\tilde{c}}\,.
$$
Then $\tilde{m}_2$ is a probability measure in $[0,t]\times\Omega$, and Young's inequality implies that
\begin{align*}
&\quad\int_{t_0}^t\into \vert Du_1-Du_2\vert  m_2(s,x)dx\,ds=\tilde{c}\int_{t_0}^t\into \vert Du_1-Du_2\vert  \tilde{m}_2(s,x)dx\,ds \smallskip \\ & \displaystyle \hspace{0.1cm}
\le\tilde{c} \left(\int_{t_0}^t\into \vert Du_1-Du_2\vert ^2 \tilde{m}_2(s,x)dx\right)^\miezz\le C\left(\int_{t_0}^t\into \vert Du_1-Du_2\vert ^2 {m}_2(s,x)dx\right)^\miezz\!\!\!,
\end{align*}
since $\tilde{c}$ is bounded thanks to the $L^p$ bound of $m_2$.

Using the Lasry-Lions monotonicity argument (see \emph{Lemma 3.1.2} of \cite{card}), we have
\begin{align*}
\displaystyle &\int_{t_0}^T\into\vert Du_1-Du_2\vert ^2(m_1(t,dx)+m_2(t,dx))dt\le \smallskip \\ & \displaystyle \hspace{2cm} \le C \into (u_1(0,x)-u_2(0,x))(m_{01}(dx)-m_{02}(dx))\smallskip \\ & \displaystyle \hspace{2cm}\le C\norm{u_1-u_2}\amu\dw(m_{01},m_{02})\,.
\end{align*}

Hence, we obtain
\begin{align*}
& \into \xi(x) (m_1(t)-m_2(t))\,dx \smallskip \\ & \displaystyle \hspace{2cm} \le C\left(\norm{u_1-u_2}\amu^\miezz\dw(m_{01},m_{02})^\miezz+\dw(m_{01},m_{02})\right)\,, \smallskip
\end{align*}
and finally, taking the sup over the $\xi$ with $\norm{\xi}\amc\le 1$ and over $t\in[0,T]$,

\begin{align}
\label{secondstep}
& \sup\limits_{t\in[0,T]}\norm{m_1(t)-m_2(t)}\amc \smallskip \\ & \displaystyle \hspace{2cm} \le C\left(\norm{u_1-u_2}\amu^\miezz\dw(m_{01},m_{02})^\miezz+\dw(m_{01},m_{02})\right)\,. \smallskip
\end{align}

Let us now call $u:=u_1-u_2$, then $u$ solves:
\begin{equation*}
\begin{cases}
\displaystyle -u_t-\mathrm{tr}(a(x)D^2 u)+V(t,x)Du=f(t,x)\,,\medskip \\ \displaystyle
u(T)=g(x)\,,\qquad
\bdone{u}\,, \medskip
\end{cases}
\end{equation*}
where
\begin{align*}
&V(t,x)=\int_0^1 H_p(t,x,\lambda Du_1(t,x)+(1-\lambda)Du_2(t,x)d\lambda\,;\\
&f(t,x)=\int_0^1\into\dm{F}(t,x,m_\lambda(t),y)(m_1(t,dy)-m_2(t,dy))d\lambda\,;\\
&g(x)=\int_0^1\into\dm{G}(x,m_\lambda(T),y)(m_1(T,dy)-m_2(T,dy))d\lambda\,,
\end{align*}
where $m_\lambda$ is defined  as follows
$$
m_\lambda(\cdot):=\lambda m_1(\cdot)+(1-\lambda)m_2(\cdot)\,.
$$
To apply Lemma \ref{sonobravo}, we estimate the regularity of the data:
\begin{align*}
\displaystyle & \sup\limits_{t\in[0,T]}\norm{f(t,\cdot)}_{\alpha} \smallskip \\
\displaystyle	& \qquad \le\sup\limits_{t\in[0,T]}\int_0^1\norm{\dm{F}(t,\cdot,m_\lambda(t),\cdot)}_{\alpha,(1+\alpha,D)}d\lambda\,\norm{m_1(t)-m_2(t)}\amc \smallskip\\
\displaystyle	& \qquad \le C\supo\norm{m_1(t)-m_2(t)}\amc\,, \smallskip
\end{align*}
analogously
\begin{align}
\norm{g(\cdot)}_{2+\alpha}\le C\supo\norm{m_1(t)-m_2(t)}\amc\,.
\end{align}
So, Eq. \eqref{espansione} implies
\begin{equation}\label{finalcountdown}
\begin{split}
\norm{u_1-u_2}\amv&\le C\supo\norm{m_1(t)-m_2(t)}\amc\,.
\end{split}
\end{equation}

Coming back to \eqref{secondstep}, we have
\begin{align*}
&\supo\norm{m_1(t)-m_2(t)}\amc\\& \hspace{.4cm}\le C\left(\left(\supo\norm{m_1(t)-m_2(t)}\amc\right)^{\miezz}\dw(m_{01},m_{02})^\miezz+\dw(m_{01},m_{02})\right)\,,
\end{align*}
hence, by a generalized Young's inequality:
\begin{align}\label{oterz}
\supo\norm{m_1(t)-m_2(t)}\amc\le C\dw(m_{01},m_{02})\,.
\end{align}
Plugging this estimate in \eqref{finalcountdown}, we finally obtain
\begin{align*}
&\norm{u_1-u_2}\amv\le C\dw(m_{01},m_{02})\label{osicond}\,.
\end{align*}
For the $L^p$ inequality, we consider $m:=m_1-m_2$. Then $m$ solves the equation
\begin{equation*}
\begin{cases}
\displaystyle		m_t-\sum\limits_{i,j}\partial^2_{ij}(a_{ij}(x)m)-\mathrm{div}(m(H_p(t,x,Du_1))=\smallskip \\
\displaystyle \hspace{1.5cm} \mathrm{div}(m_2(H_p(t,x,Du_2)-H_p(t,x,Du_1)))\,, \medskip \\

\displaystyle		m(t_0)=m_{01}-m_{02}\,, \medskip \\

\displaystyle		\left[m_1 - m_2 \right]_{\vert \partial\Omega}=0\,, \medskip
\end{cases}
\end{equation*}
i.e. $m$ is a solution of \eqref{linfp} with $f=\mathrm{div}(m_2(H_p(t,x,Du_2)-H_p(t,x,Du_1)))$, $\mu_0=m_{01}-m_{02}$, $b=H_p(t,x,Du_1)$. Then estimates \eqref{stimefokker} imply
$$
\norm{m_1-m_2}_{L^p(Q_T)}\le C\left(\norm{\mu_0}\amb+\norm{f}\amf\right)\,.
$$
We estimate the right-hand side term. As regards $\mu_0$ we have
$$
\norm{\mu_0}_{-(1+\alpha)}=\sup\limits_{\norm{\phi}_{1+\alpha}\le 1}\into \phi(x)(m_{01}-m_{02})(dx)\le C\dw(m_{01},m_{02})\,.
$$
For the $f$ term, we argue in the following way:
\begin{align*}
\norm{f}\amf&=\int_0^T\sup\limits_{\norm{\phi}_{W^{1,\infty}}\le 1}\left(\into H_p(x,Du_2)-H_p(x,Du_1)D\phi\,m_2(t,dx)\right)\,dt\\&\le C\norm{u_1-u_2}\amu\le C\dw(m_{01},m_{02})\,,
\end{align*}
which allows us to conclude.
\end{proof}
\end{proposition}
The previous proposition \eqref{estimates} allows us to state that
\begin{equation}\label{firstmaster}
\begin{split}
&\sup\limits_{t\in[0,T]}\sup\limits_{m\in\mathcal{P}^{sub}(\Omega)}\norm{U(t,\cdot,m)}_{2+\alpha}\le C\,,\\
&\sup\limits_{t\in[0,T]}\sup_{m_1\neq m_2}\left[\left(\dw(m_1,m_2)\right)^{-1}\norm{U(t,\cdot,m_1)-U(t,\cdot,m_2)}_{2+\alpha}\right]\le C\,,
\end{split}
\end{equation}
which are two initial regularity results for the function $U$.

\section{Linearized system and differentiability of {\it U} with respect to the measure}
\label{sec5}
This section is devoted to the study of the following \emph{linearized MFG system}:
\begin{equation}\label{linear}
\begin{cases}
\displaystyle -z_t-\mathrm{tr}(a(x)D^2z)+H_p(t,x,Du)Dz= {\dm{F}}(t,x,m(t))(\rho(t))+h(t,x)\,, \medskip \\
\displaystyle	\rho_t-\sum\limits_{ij}\partial^2_{ij}(a_{ij}(x)\rho)  -\mathrm{div}(\rho(H_p(t,x,Du))) + \\
\displaystyle \hspace{5cm} -\mathrm{div}(m H_{pp}(t,x,Du) Dz+c)=0 , \medskip \\
\displaystyle z(T,x)= {\dm{G}}(x,m(T))(\rho(T))+z_T(x)\,,\qquad\rho(t_0)=\rho_0\,, \medskip\\
\displaystyle \bdone{z}\,,\qquad\bdone{\rho}\, \medskip ,
\end{cases}
\end{equation}
where $z_T\in\mathcal{C}^{2+\alpha,D},\quad\rho_0\in\mathcal{C}^{-(1+\alpha)}$, $h\in \mathcal{C}^{0,\alpha}([t_0,T]\times\Omega)$, $c\in L^1([t_0,T]\times\Omega)$,
and where we define for $F$ (and for $G$)
$$
{\dm{F}}(t,x,m(t))(\rho(t)):=\left\langle{\dm{F}}(t,x,m(t),\cdot),\rho(t)\right\rangle\,,
$$
where the duality is between $\mathcal C^{-(1+\alpha),D}$ and $\mathcal C^{(1+\alpha),D}$.

The study of this system plays a crucial role in proving the $\mathcal{C}^1$ character of $U$ in terms of $m$. In particular, if we define the couple $(v,\mu)$ as the solution of \eqref{linear} with $h=c=z_T=0$ and $\mu(t_0)=\mu_0$, we obtain
\begin{equation}\label{reprform}
v(t_0,x)=\left\langle\frac{\delta U}{\delta m}(t_0,x,m_0,\cdot),\mu_0\right\rangle\,. 
\end{equation}
Let us start by giving a suitable definition of a solution for the previous system.
\begin{definition}\label{definition}
A couple $(z,\rho)$ is a solution of \eqref{linear} if $z\in\mathcal{C}^{1,2+\alpha}$, is a classical solution of the linear HJB equation and $\rho\in \mathcal{C}([0,T];\mathcal{C}^{-(1+\alpha),D}(\Omega))\cap L^1(Q_T)$ solves the Fokker-Planck equation, accordingly with Definition \ref{canzonenuova}.
\end{definition}

Let us underline that $c\in L^1([t_0,T]\times\Omega)\implies \mathrm{div}(c)\in L^1([t_0,T];W^{-1,\infty}(\Omega))$. Therefore, the well-posedness of the  Fokker-Planck equation is included in Proposition \ref{peggiodellagerma}.

The following Proposition states existence, uniqueness and regularity for the problem defined by Eq. \eqref{linear}. The proof relies on results previously obtained for the Fokker-Planck equation, 
and it proceeds similarly to the Neumann case \cite{Ricciardi}. We report it here for the sake of completeness.

\begin{proposition}\label{linearD}
Let hypotheses \ref{ipotesi} hold for a certain $0<\alpha<1$, and let $\rho_0\in\mathcal C^{-1}$. Then there exists a unique solution $(z,\rho)\in\mathcal{C}^{1,2+\alpha}\times\,\left(\mathcal{C}([0,T];\mathcal{C}^{-(1+\alpha),D}(\Omega))\cap L^1(Q_T)\right)$ for the system \eqref{linear}, satisfying  
\begin{equation}
\begin{split}\label{stimelin}
\norm{z}\amv+\sup\limits_t\norm{\rho(t)}_{-(1+\alpha),D}+\norm{\rho}_{L^p}\le CM\spazio,
\end{split}
\end{equation}
where $C$ depends on $H$, $p$ is defined as in section \ref{subsection}, and $M$ is given by
\begin{equation}\label{emme}
M:=\norm{z_T}_{2+\alpha}+\norm{\rho_0}_{-(1+\alpha)}+\norm{h}_{0,\alpha}+\norm{c}_{L^1}\,.
\end{equation}
Moreover, the solution $(v,\mu)$ related to $h=c=z_T=0$ satisfies
\begin{equation}\label{sbrigati}
\norm{v}_{1,2+\alpha}+\supo\norm{\mu(t)}_{-(2+\alpha),D}+\norm{\mu}\aam\le C\norm{\mu_0}_{-(2+\alpha)}\,.
\end{equation}
\end{proposition}

Observe that the analogue of this Theorem in the Neumann case is given by \cite[Propositions 5.8,5.11]{Ricciardi}. However, even in this case, we have refined the estimate \eqref{sbrigati} in order to adjust the issues raised in \cite{jakob}. See also Remark \ref{rem:espen}.
\begin{proof}
As always, we can assume $t_0 = 0$ without loss of generality. The main idea is to apply Schaefer's Theorem.

\emph{Step 1: Definition of the map $\mathbf{\Phi}$ satisfying Schaefer's Theorem}.
We set $X := \mathcal{C}([0,T]; \mathcal{C}^{-(1+\alpha),D})$, endowed with the norm
\begin{equation*}
\norm{\phi}_X := \sup_{t \in [0,T]} \norm{\phi(t)}_{-(1+\alpha),D}\,.
\end{equation*}

For $\rho \in X$, we consider the classical solution $z$ of the following equation
\begin{equation}
\label{zlin}
\begin{cases}
-z_t - \tr z + H_p(t,x, Du)Dz = \mathlarger{\dm{F}}(t,x, m(t))(\rho(t)) + h(t,x)\,,\\
z(T) = \mathlarger{\dm{G}}(x, m(T))(\rho(T)) + z_T\,,\\
z|_{\partial\Omega} = 0\,.
\end{cases}
\end{equation}

We note that, from Hypotheses \ref{ipotesi}, we have
$$\frac{\delta F}{\delta m}(t, x, m(t), y)|_{y \in \partial\Omega} = 0, \quad \frac{\delta G}{\delta m}(x, m(T), y)|_{y \in \partial\Omega} = 0$$
Hence, compatibility conditions are satisfied for equation \eqref{zlin} and, from \emph{Theorem 5.1.21} of \cite{lunardi}, $z$ satisfies
\begin{equation}\label{stimz}
\begin{split}
\norm{z}_{1,2+\alpha} &\leq C\left(\norm{z_T}_{2+\alpha} + \sup_{t \in [0,T]} \norm{\rho(t)}_{-(2+\alpha),D} + \norm{h}_{0,\alpha}\right)\\
&\leq C\left(M + \sup_{t \in [0,T]} \norm{\rho(t)}_{-(1+\alpha),D}\right)\,,
\end{split}
\end{equation}
where we also use hypothesis $(vii)$ of \ref{ipotesi}, for the boundary condition of $\dm{G}$.

Then we define $\mathbf{\Phi}(\rho) := \tilde{\rho}$, where $\tilde{\rho}$ is the solution in the sense of Definition \ref{canzonenuova} to:
\begin{equation}
\label{plin}
\begin{cases}
\tilde{\rho}_t - \mathrm{div}(a(x)D\tilde{\rho}) - \mathrm{div}(\tilde{\rho}(H_p(t,x,Du))) - \mathrm{div}(mH_{pp}(t,x,Du)Dz + c) = 0\\
\tilde{\rho}(0) = \rho_0\\
\tilde{\rho}|_{\partial\Omega} = 0
\end{cases}\,.
\end{equation}

Thanks to Proposition \ref{peggiodellagerma}, we have $\tilde{\rho} \in X$. We want to prove that the map $\mathbf{\Phi}$ is continuous and compact.

For the compactness, let $\{\rho_n\}_n \subset X$ be a subsequence with $\norm{\rho_n}_X \leq C$ for a certain $C > 0$. We consider for each $n$ the solutions $z_n$ and $\tilde{\rho}_n$ of \eqref{zlin} and \eqref{plin} associated to $\rho_n$.

Using \eqref{stimz}, we have $\norm{z_n}_{1,2+\alpha} \leq C_1$, where $C_1$ depends on $C$. Then, thanks to Ascoli-Arzelà's Theorem, $\exists z$ s.t. $z_n \to z$ up to subsequences at least in $\mathcal{C}([0,T]; \mathcal{C}^1(\Omega))$.

Using the pointwise convergence of $Dz_n$ and the $L^p$ boundedness of $m$ stated in Proposition \ref{estimates}, we immediately obtain
$$
mH_{pp}(t,x,Du)Dz_n + c \to mH_{pp}(t,x,Du)Dz + c \quad \text{in } L^1(Q_T)\,,
$$
which immediately implies
$$
\mathrm{div}(mH_{pp}(t,x,Du)Dz_n + c) \to \mathrm{div}(mH_{pp}(t,x,Du)Dz + c) \quad \text{in } L^1(W^{-1,\infty})\,.
$$

Hence, stability results proved in Proposition \ref{peggiodellagerma} prove that $\tilde{\rho}_n \to \tilde{\rho}$ in $X$, where $\tilde{\rho}$ is the solution related to $Dz$. This proves the compactness result.

The continuity of $\Phi$ can be proved using the same computations as for the compactness.

Finally, in order to apply Schaefer's theorem, we have to prove that
$$
\exists M > 0 \text{ s.t. } \rho = \sigma\mathbf{\Phi}(\rho) \text{ and } \sigma \in [0,1] \Rightarrow \norm{\rho}_X \leq M\,.
$$

We will prove in the next step that, if $\rho = \sigma\mathbf{\Phi}(\rho)$, then the couple $(z,\rho)$ satisfies \eqref{stimelin}. This allows us to apply Schaefer's theorem and also gives us the desired estimate \eqref{stimelin}, since each solution $(z,\rho)$ of the system satisfies $\rho = \sigma\mathbf{\Phi}(\rho)$ with $\sigma = 1$.

\emph{Step 2: Estimate of $\rho$ and $z$}. Let $(\rho,\sigma) \in X \times [0,1]$ such that $\rho = \sigma\mathbf{\Phi}(\rho)$. Then the couple $(z,\rho)$ satisfies
\begin{equation*}
\begin{cases}
-z_t - \mathrm{tr}(a(x)D^2z) + H_p(t,x,Du)Dz = \mathlarger{\dm{F}}(t,x,m(t))(\rho(t)) + h(t,x)\\
\rho_t - \mathrm{div}(a(x)D\rho) - \mathrm{div}(\rho H_p(t,x,Du)) - \sigma\mathrm{div}(mH_{pp}(t,x,Du)Dz + c) = 0\\
z(T,x) = \mathlarger{\dm{G}}(x,m(T))(\rho(T)) + z_T(x) \quad \rho(0) = \sigma\rho_0\\
z|_{\partial\Omega} = 0 \quad \rho|_{\partial\Omega} = 0
\end{cases}\,.
\end{equation*}

We want to use $z$ as a test function for the equation of $\rho$. This is allowed since $z$ satisfies \eqref{hjbfp} with
\begin{align*}
\psi = \dm{F}(t,x,m(t))(\rho(t)) + h(t,x) \in L^\infty(\Omega), \quad \xi = \dm{G}(x,m(T))(\rho(T)) + z_T(x) \in \mathcal{C}^{1+\alpha,D}
\end{align*}

From the weak formulation of $\rho$, we obtain:
\begin{equation*}
\begin{split}
&\int_\Omega \left(\rho(T,x)z(T,x) - \sigma\rho_0(x)z(0,x)\right)dx = -\sigma\int_0^T\int_\Omega \langle c, Dz\rangle dxdt\\
&-\int_0^T\int_\Omega \rho(t,x)\left(\dm{F}(t,x,m(t))(\rho(t)) + h\right)dxdt - \sigma\int_0^T\int_\Omega m\langle H_{pp}(t,x,Du)Dz, Dz\rangle dxdt\,.
\end{split}
\end{equation*}

Using the terminal condition of $z$ and the monotonicity of $F$ and $G$ $(vi)$ in \ref{ipotesi}, we get a first estimate:
\begin{equation}\label{stimasigma}
\begin{split}
\sigma\int_0^T\int_\Omega m\langle H_{pp}(t,x,Du)Dz, Dz\rangle dxdt
\leq &\sup_{t\in[0,T]}\norm{\rho(t)}_{-(2+\alpha),D}\norm{z_T}_{2+\alpha} + \norm{\rho}_{L^p}\norm{h}_\infty\\
&+ \norm{z}_{1,2+\alpha}\left(\norm{\rho_0}_{-(2+\alpha),D} + \norm{c}_{L^1}\right)\\
\leq &M\left(\sup_{t\in[0,T]}\norm{\rho(t)}_{-(1+\alpha),D} + \norm{\rho}_{L^1} + \norm{z}_{1,2+\alpha}\right)\,.
\end{split}
\end{equation}

We already know an initial estimate on $z$ in \eqref{stimz}. Now we need to estimate $\rho$.

Using \eqref{stimefokker} from Proposition \ref{peggiodellagerma} we obtain
\begin{equation}\label{duality}
\sup_{t\in[0,T]}\norm{\rho}_{-(1+\alpha),D} + \norm{\rho}_{L^p} \leq C\left(\norm{\sigma mH_{pp}(t,x,Du)Dz}_{L^1} + \norm{c}_{L^1} + \norm{\rho_0}_{-(1+\alpha)}\right)
\end{equation}

As regards the first term in the right-hand side, we can use Hölder's inequality and \eqref{stimasigma} to obtain
\begin{align*}
&\norm{mH_{pp}(t,x,Du)Dz}_{L^1} = \sigma\sup_{\substack{\norm{\phi}_\infty \leq 1\\\phi\in L^\infty(Q_T;\mathbb{R}^d)}}\int_0^T\int_\Omega m\langle H_{pp}(t,x,Du)Dz, \phi\rangle dxdt\\
&\leq \sigma\left(\int_0^T\int_\Omega m\langle H_{pp}(t,x,Du)Dz, Dz\rangle dxdt\right)^{1/2}\left(\int_0^T\int_\Omega m\langle H_{pp}(t,x,Du)\phi, \phi\rangle dxdt\right)^{1/2}\\
&\leq M^{1/2}\left(\sup_{t\in[0,T]}\norm{\rho(t)}^{1/2}_{-(1+\alpha),D} + \norm{\rho}^{1/2}_{L^1} + \norm{z}^{1/2}_{1,2+\alpha}\right)\,
\end{align*}

Putting these estimates into \eqref{duality} we obtain
\begin{align*}
\sup_{t\in[0,T]}\norm{\rho}_{-(1+\alpha),D} + \norm{\rho}_{L^p} \leq C\left(M + M^{1/2}\left(\sup_{t\in[0,T]}\norm{\rho(t)}^{1/2}_{-(1+\alpha),D} + \norm{\rho}^{1/2}_{L^1} + \norm{z}^{1/2}_{1,2+\alpha}\right)\right)\,.
\end{align*}

Using a generalized Young's inequality with suitable coefficients, we get
\begin{align}\label{stimarho}
\sup_{t\in[0,T]}\norm{\rho}_{-(1+\alpha),D} + \norm{\rho}_{L^p} \leq C\left(M + M^{1/2}\norm{z}^{1/2}_{1,2+\alpha}\right)\,.
\end{align}

This gives us an initial estimate for $\rho$, depending on the estimate of $z$.

Coming back to \eqref{stimz}, \eqref{stimarho} implies
\begin{align*}
\norm{z}_{1,2+\alpha} \leq C\left(M + M^{1/2}\norm{z}^{1/2}_{1,2+\alpha}\right)\,.
\end{align*}

Using a generalized Young's inequality with suitable coefficients, this implies
$$
\norm{z}_{1,2+\alpha} \leq CM\,.
$$

Plugging this estimate in \eqref{stimarho}, we finally obtain
\begin{align*}
\norm{z}_{1,2+\alpha} + \sup_{t\in[0,T]}\norm{\rho}_{-(1+\alpha),D} + \norm{\rho}_{L^p} \leq CM\,.
\end{align*}

This concludes the existence part.

\emph{Step 3. Uniqueness}. Let $(z_1,\rho_1)$ and $(z_2,\rho_2)$ be two solutions of \eqref{linear}. Then the couple $(z,\rho) := (z_1-z_2, \rho_1-\rho_2)$ satisfies the following linear system:
\begin{equation*}
\begin{cases}
-z_t - \mathrm{tr}(a(x)D^2z) + H_p(t,x,Du)Dz = \mathlarger{\dm{F}}(t,x,m(t))(\rho(t)) = 0\,,\\
\rho_t - \mathrm{div}(a(x)D\rho) - \mathrm{div}(\rho H_p(t,x,Du)) - \mathrm{div}(mH_{pp}(t,x,Du)Dz) = 0\,,\\
z(T,x) = \mathlarger{\dm{G}}(x,m(T))(\rho(T))\,, \quad \rho(t_0) = 0\,,\\
z|_{\partial\Omega} = 0\,, \quad \rho|_{\partial\Omega} = 0\,,
\end{cases}
\end{equation*}

i.e., a system of the form \eqref{linear} with $h = c = z_T = \rho_0 = 0$. Then estimation \eqref{stimelin} tells us that
$$
\norm{z}_{1,2+\alpha} + \sup_{t\in[0,T]}\norm{\rho}_{-(1+\alpha),D} + \norm{\rho}_{L^p} \leq 0\,,
$$
and so $z = 0$, $\rho = 0$. This proves the estimate \eqref{stimelin}.

Concerning the second inequality \eqref{sbrigati}, we consider the solution $(v,\mu)$ obtained above. Since $\mu$ satisfies $\mu = \sigma\Phi(\mu)$ with $\sigma = 1$, we can use \eqref{stimz} with $z_T = h = 0$ and obtain
\begin{equation}\label{cumnupnat}
\norm{v}_{1,2+\alpha} \leq C\sup_{t\in[0,T]}\norm{\mu(t)}_{-(2+\alpha),D}\,.
\end{equation}

We want to estimate the right-hand side. Using Corollary \ref{corollary} (which addresses the regularity issues raised in \cite{jakob}) we have
\begin{equation}\label{ngroc}
\sup_{t\in[0,T]}\norm{\mu(t)}_{-(2+\alpha),D}+\norm{\mu}_{-\left(\frac\alpha2,\alpha\right)} \leq C\left(\norm{\mu_0}_{-(2+\alpha)} + \norm{mH_{pp}(t,x,Du)Dv}_{L^1}\right)\,.
\end{equation}

The last term is estimated, as above, by
\begin{equation}\label{mannaggia}
\norm{\sigma mH_{pp}(t,x,Du)Dv}_{L^1} \leq C\left(\int_0^T\int_\Omega m\langle H_{pp}(t,x,Du)Dv, Dv\rangle dxdt\right)^{1/2}\,.
\end{equation}

The right-hand side term can be bounded using \eqref{stimasigma} with $h = z_T = c = 0$:
\begin{equation}\label{crist}
\begin{split}
\int_0^T\int_\Omega m\langle H_{pp}(t,x,Du)Dv, Dv\rangle dxdt
\leq \norm{v}_{1,2+\alpha}\norm{\mu_0}_{-(2+\alpha)}\,.
\end{split}
\end{equation}

Hence, plugging estimates \eqref{mannaggia} and \eqref{crist} into \eqref{ngroc} we obtain
\begin{equation}\label{probbiatottquanta}
\sup_{t\in[0,T]}\norm{\mu(t)}_{-(2+\alpha),D}+\norm{\mu}_{-\left(\frac\alpha2,\alpha\right)}  \leq C\left(\norm{\mu_0}_{-(2+\alpha)} + \norm{v}_{1,2+\alpha}^{1/2}\norm{\mu_0}_{-(2+\alpha),D}^{1/2}\right)\,.
\end{equation}

Coming back to \eqref{cumnupnat} and using a generalized Young's inequality, we get
$$
\norm{v}_{1,2+\alpha} \leq C\norm{\mu_0}_{-(2+\alpha)}\,,
$$
and finally, substituting the last estimate into \eqref{probbiatottquanta}, we obtain \eqref{sbrigati} and conclude.
\end{proof}

Throughout the rest of the paper, we will denote with $(v,\mu)$ a solution to the system \eqref{linear}, with $h=c=z_T=0$ and  $\mu_0:=\rho_0$. We will refer to this system as the \emph{pure linearized system.} Instead, the general system \eqref{linear}, with solution $(z,\rho)$, will be called the \emph{general linearized system}.

To prove $U \in \mathcal{C}^1$ as well as the related representation formula \eqref{reprform}, we have to prove that the pure linearized system has a fundamental solution, which is the content of the next proposition.

\begin{proposition}
Let hypotheses \ref{ipotesi} hold. Then there exists a function $K:[0,T]\times\Omega\times\mathcal{P}^{sub}(\Omega)\times\Omega\to\R$ such that, for any solution $(v,\mu)$, with initial data $(t_0,m_0,\mu_0)$, we have
\begin{equation}\label{repres}
v(t_0,x)=_{-(1+\alpha)}\!\!\langle \mu_0,K(t_0,x,m_0,\cdot)\rangle_{1+\alpha}
\end{equation}
Moreover, $K$ is twice differentiable with respect to both $x$ and $y$, and it holds
\begin{equation}\label{kappa}
\sup\limits_{(t,m)\in[0,T]\times\mathcal{P}^{sub}(\Omega)}\norm{K(t,\cdot,m,\cdot)}_{2+\alpha,2+\alpha}\le C\,,
\end{equation}

\begin{proof}
Let $\mu_0=\delta_y$ be the Dirac function at $y\in\Omega$. We define $K(t_0,x,m_0,y)=v(t_0,x;\delta_y)$, where $v(\cdot,\cdot;\mu_0)$ indicates the function $v$ related to the pure linearized system with initial data $\mu_0$. Exploiting \eqref{stimelin}, we have that $K$ is twice differentiable w.r.t. $x$, and it holds
$$
\norm{K(t_0,\cdot,m_0,y)}_{2+\alpha}\le C\norm{\delta_y}_{-(1+\alpha)}=C
$$
Furthermore, the linearity character of \eqref{linear} implies
$$
\frac{K(t_0,x,m_0,y+he_i)-K(t_0,x,m_0,y)}h=v(t_0,x;\Delta_h^i\delta_{y})\,,
$$
where we use the notation $\Delta_h^i\delta_{y}=\frac1h(\delta_{y+he_i}-\delta_y)$. The linear character of the pure linearized system directly implies that the solution $(v,\mu)$ is stable with respect to the initial condition $\mu_0$, allowing to pass to the limit and obtain
$$
\frac{\partial K}{\partial y_i}(t_0,x,m_0,y)=v(t_0,x;-\partial_{y_i}\delta_y)\,,
$$
with the derivative that has to be intended in a distributional sense. As to prove the existence and second derivatives' bounds, we consider the incremental ratio
\begin{equation}\label{alotteriarubabbeh}
R^h_{i,j}(x,y):=\frac{\partial_{y_i}K(t_0,x,m_0,y+he_j)-\partial_{y_i}K(t_0,x,m_0,y)}{h}\,.
\end{equation}
Hence, estimate \eqref{sbrigati} together with Lagrange Theorem imply that,
for $\vert l\vert \le 2,$ we have
\begin{small}
\begin{align*}	&\hspace{0.9cm}\left\vert D^l_xR^h_{i,j}(x,y)-D^l_xR^k_{i,j}(x,y)\right\vert =\left\vert D^l_xv\left(t_0,x;\Delta_h^j(-\partial_{y_i}\delta_y)-\Delta_k^j(-\partial_{y_i}\delta_y)\right)\right\vert \\&\le C\norm{\Delta_h^j(-\partial_{y_i}\delta_y)-\Delta_k^j(-\partial_{y_i}\delta_y)}_{-(2+\alpha)}\\&=\sup\limits_{\norm{\phi}_{2+\alpha}\le 1}\left(\frac{\partial_{y_i}\phi(y+he_j)-\partial_{y_i}\phi(y)}{h}-\frac{\partial_{y_i}\phi(y+ke_j)-\partial_{y_i}\phi(y)}{k}\right)\\&=\sup\limits_{\norm{\phi}_{2+\alpha}\le 1}\left(\partial^2_{y_iy_j}\phi(y_{\phi,h})-\partial^2_{y_iy_j}\phi(y_{\phi,k})\right)\le\sup\limits_{\norm{\phi}_{2+\alpha}\le 1}\vert y_{\phi,h}-y_{\phi,k}\vert ^\alpha\le \vert h\vert ^\alpha+\vert k\vert ^\alpha,
\end{align*}
\end{small}
for a certain $y_{\phi,h}$, resp. $y_{\phi,k}$ in the line segment between $y$ and $y+he_j$ (resp. $y+ke_j$), and where we have used the same notation as seen above for $\Delta_h^j(-\partial_{y_i}\delta_y)$.

The latter proves that \eqref{alotteriarubabbeh}, together with its first and second derivative w.r.t $x$, are Cauchy sequences in $h$, implying that  $D^l_x\dm{U}$ is twice differentiable w.r.t. $y$ and for all $\vert l\vert \le2\,$.

As to conclude estimate \eqref{kappa}, we prove the estimate for the second order derivatives w.r.t. y, the first derivative being simpler. By estimate \eqref{sbrigati} on $v$,
if  $y,y'\in\Omega$, then
\begin{small}
\begin{align*}	\norm{R^h_{i,j}(\cdot,y)-R^h_{i,j}(\cdot,y')}_{2+\alpha}\!\!&=\norm{v\big(t_0,\cdot;\Delta^j_h(-\partial_{y_i}\delta_y+\partial_{y_i}\delta_{y'})\big)}_{2+\alpha}\\&\le C\norm{\Delta^j_h(-\partial_{y_i}\delta_y+\partial_{y_i}\delta_{y'})}_{-(2+\alpha)}.
\end{align*}
\end{small}

Passing to the limit for $h\to 0$, we have
$$
\Delta^j_h(-\partial_{y_i}\delta_y+\partial_{y_i}\delta_{y'})\longrightarrow\partial_{y_j}\partial_{y_i}\delta_y-\partial_{y_j}\partial_{y_i}\delta_{y'} \,,\,\mbox{in }\mathcal{C}^{-(2+\alpha)}\,.
$$

Then, by Ascoli-Arzel\`{a} and previously obtained convergence result for $R^h_{i,j}$, we have
\begin{align*}
&\norm{\partial^2_{y_iy_j}\left(K(t_0,\cdot,m_0,y)-K(t,\cdot,m_0,y')\right)}_{2+\alpha} \\ & \displaystyle \hspace{2cm} \le C\norm{\partial_{y_iy_j}^2(\delta_y-\delta_{y'})}_{-(2+\alpha)}\le C\vert y-y'\vert ^\alpha\,,
\end{align*}
which proves \eqref{kappa}. The representation formula \eqref{repres} is an immediate consequence of the linear character of the equation and of the density of the Dirac functions generated set, hence concluding the proof.
\end{proof}
\end{proposition}
We are now ready to consider the main topic of this section: we want to prove that the function $K$ is actually the derivative of $U$ with respect to the measure. Let us underline that the differentiability with respect to the measure $m$ will be the key for proving $U$ is indeed a classical solution of the Master Equation \eqref{Master}.

\begin{theorem}
Suppose hypotheses \ref{ipotesi} hold. Then the function $U$ defined by \eqref{eq:u} is differentiable with respect to $m$, with the derivative given by
$$
\dm{U}(t,x,m,y)=K(t,x,m,y)\,.
$$
\begin{proof}
Within the present proof, given two functions $a_1,a_2$, we define $a_{1+\tau}:=\tau a_2+(1-\tau)a_1$, for $\tau\in[0,1]$.

We will prove a more general fact, the representation formula for $U$ being then a direct consequence of it.

In particular, if $(u_1,m_1)$ and $(u_2,m_2)$ are two solutions of \eqref{meanfieldgames}, with initial conditions $(t_0,m_{01})$ and $(t_0,m_{02})$, and $(v,\mu)$ is the solution of the pure linearized system related to $(u_2,m_2)$, with initial condition $(t_0,m_{01}-m_{02})$, then a sort of first-order Taylor expansion of $U$ with respect to $m$ holds, namely:
\begin{equation}\label{boundmder}
\norm{u_1-u_2-v}\amv+\supo\norm{m_1(t)-m_2(t)-\mu(t)}\amc\le C\dw(m_{01},m_{02})^{2}\,.
\end{equation}

As to prove above {\it expansion}, let us start defining $(z,\rho)=(u_1-u_2-v,m_1-m_2-\mu)$. Then $(z,\rho)$ satisfies the general linearized system \eqref{linear} related to $(u_2,m_2)$, with data $h=h_1+h_2$, $c=c_1+c_2$ and $z_T$ given by:
\begin{align*}
&h_1=-\int_0^1 (H_p(t,x,Du_{1+s})-H_p(t,x,Du_2))\cdot D(u_1-u_2)\,ds\,, \smallskip \\
&h_2=\int_0^1\!\into\left(\dm{F}(t,x,m_{1+s}(t),y)-\dm{F}(t,x,m_2(t),y)\right)(m_1(t)-m_2(t))(dy)ds, \smallskip \\
&c_1(t)=(m_1(t)-m_2(t))H_{pp}(t,x,Du_2)(Du_1-Du_2)\,,\smallskip\\
&c_2(t)=m_1\int_0^1\left(H_{pp}(t,x,Du_{1+s})-H_{pp}(t,x,Du_2)\right)(Du_1-Du_2)\,ds\,,\smallskip \\
&z_T=\int_0^1\into\left(\dm{G}(x,m_{1+s}(T),y)-\dm{G}(x,m_2(T),y)\right)(m_1(T)-m_2(T))(dy)ds\,. \smallskip
\end{align*}
Applying \eqref{stimelin} one has
\begin{equation}\label{rogueuno}
\norm{z}\amv+\supo\norm{\rho(t)}\amc\le C\left(\norm{h}_{0,\alpha}+\norm{c}_{L^1}+\norm{z_T}_{2+\alpha}\right)\,.
\end{equation}

We want to estimate the right-hand side term  to obtain the desired Taylor expansion. As regards $h$, exploiting 
eq. \eqref{lipsch} and H\"{o}lder norm properties, we have:

\begin{align*}
\norm{h_1}_{0,\alpha}&=\norm{\int_0^1\int_0^1 s\, \langle H_{pp}(t,x,Du_{1+rs})\,(Du_1-Du_2)\,,\, (Du_1-Du_2)\rangle\,drds}_{0,\alpha}\\&\le C\dw(m_{01},m_{02})^2 \,,
\end{align*}
Analogously, again by eq. \eqref{lipsch} and exploiting  regularity of both $F$ and $G$, the same estimate also holds for $h_2$ and $z_T$. 
For the function $c$, we can write
\begin{align*}
\norm{c_1}_{L^1}=\intif H_{pp}(t,x,Du_2)(Du_1-Du_2)(m_1(t,dx)-m_2(t,dx))\,dt\\ \le C\norm{u_1-u_2}\amv\norm{m_1(t)-m_2(t)}\le C\dw(m_{01},m_{02})^{2}\,,
\end{align*}
then, proceeding as above, we have

\begin{align*}
& \norm{c_2}_{L^1}= \\ & \qquad \int_0^1\intif \left(H_{pp}(t,x,Du_{1+s})-H_{pp}(t,x,Du_2)\right)(Du_1-Du_2)m_1(t,dx)\,dtds\\
& \quad \le \,C\norminf{Du_1-Du_2}^2\le C\dw(m_0^1,m_0^2)^2\,.
\end{align*}
Substituting these estimates in \eqref{rogueuno}, we obtain \eqref{boundmder}. Using the representation \eqref{repres} for $v$, we get

$$
\begin{array}{c}\displaystyle	\norminf{U(t_0,\cdot,m_{01})-U(t_0,\cdot,m_{02})-\into K(t_0,\cdot,m_{02},y)(m_{01}-m_{02})(dy) } \medskip \\ 
\displaystyle \hspace{7cm} \le C\dw(m_{01},m_{02})^{2}\,. \medskip	\end{array}
$$

As a straightforward consequence, we have that $U$ is differentiable with respect to $m$ and
$$
\dm{U}(t,x,m,y)=K(t,x,m,y)\,,
$$
hence concluding the proof.
\end{proof}
\end{theorem}
Using \eqref{kappa} we also obtain the following strong regularity estimate for $\dm{U}$:
\begin{equation}\label{regdu}
\sup\limits_t\norm{\dm{U}(t,\cdot,m,\cdot)}_{2+\alpha,2+\alpha}\le C\,.
\end{equation}

This gives sense to the Master Equation \eqref{Master}, since the quantity $D_yD_mU$ is now well-defined. However, to apply Theorem \ref{settepuntouno} we still need to prove the continuity of $D_yD_mU$ in the measure variable.

In the next result we establish a Lipschitz bound for $\dm{U}$ with respect to the measure $m$. This estimate plays a crucial role, as it yields the continuity of $D_yD_mU$ with respect to $m$, a property that is essential for proving the existence of solutions to the Master Equation in the appropriate framework. Such regularity of $D_yD_mU$ was not available in the Neumann case, as reported in \cite{}. We also point out that, by adapting the same ideas, an analogous result could be obtained in the Neumann setting as well, although we shall not pursue this direction in the present work.

\begin{theorem} \label{MEderivativeregularity}
Let Assumptions \ref{ipotesi} hold. Then the derivative of the solution of the Master Equation $\dm{U}$ is Lipschitz continuous with respect to the measure $m$:

\begin{equation}
\label{lipsch1}
\sup_{t \in [0,T]} \sup_{{m}_1 \neq {m}_2} ( \mathbf{d}_1 ({m}_1, {m}_2))^{-1}  \Bigg\vert  \Bigg\vert   \dm{U} (t, \cdot, {m}_1, \cdot )-  \dm{U} (t, \cdot, {m}_2, \cdot ) \Bigg\vert  \Bigg\vert _{2+\alpha, 2+ \alpha} \, \leq \, C
\end{equation}
where $C$ depends on $d$, $F$, $G$, $H$ and $T$. 
\end{theorem}
\begin{proof}
We consider, for $i=1,2$, the solution $(v_i,\mu_i)$ of the linearized system \eqref{linear} related to $(u_i,m_i)\,.$

To avoid too heavy notations, we take $t_0=0$ and we define
\begin{align*}
\begin{array}{ll}
H'_i(t,x):=H_p(t,x,Du_i(t,x))\,, & H''_i(t,x)=H_{pp}(t,x,Du_i(t,x))\,,\\
F'(t,x,m,\mu)=\displaystyle\into \dm{F}(t,x,m,y)\,\mu(dy)\,, & G'(x,m,\mu)=\displaystyle\into \dm{G}(x,m,y)\,\mu(dy)\,. 
\end{array}
\end{align*}

Then the couple $(z,\rho):=(v_1-v_2,\mu_1-\mu_2)$ satisfies the following linear system:
$$
\begin{cases}
-z_t-\mathrm{tr}(a(x)D^2z)+H'_1\cdot Dz=F'(t,x,m_1(t),\rho(t))+h\,,\\
\rho_t-\sum\limits_{i,j}\partial^2_{ij}(a_{ij}(x)\rho)-\mathrm{div}(\rho H'_1)-\mathrm{div}(m_1H''_1Dz+c)=0\,,\\
z(T,x)=G'(x,m_1(T),\rho(T))+z_T\,,\qquad \rho(t_0)=0\,,\\
z_{|\partial\Omega}=0\,,\qquad \rho_{|\partial\Omega}=0\,,
\end{cases}
$$
where
\begin{align*}
&h(t,x)=h_1(t,x)+h_2(t,x)\,,\\
&h_1(t,x)=F'(t,x,m_1(t),\mu_2(t))-F'(t,x,m_2(t),\mu_2(t))\,,\\
&h_2(t,x)=\big(H'_1(t,x)-H'_2(t,x)\big)\cdot Dv_2(t,x)\,,\\
&c(t,x)=\mu_2(t)\big(H'_1-H'_2)(t,x)+\big[(m_1H''_1-m_2H''_2)\big](t,x)\,,\\
&z_T(x)=G'(x,m_1(T),\mu_2(T))-G'(x,m_2(T),\mu_2(T))\,.
\end{align*}

Applying \eqref{stimelin} we obtain this estimate on $z$:
$$
\norm{z}\amv\le C\left(\norm{z_T}_{2+\alpha}+\norm{h}_{0,\alpha}+\norm{c}_{L^1}\right)\,.
$$

Now we estimate the terms in the right-hand side.

The term with $z_T$, thanks to \eqref{sbrigati} and the hypothesis \emph{(v)} of \ref{ipotesi}, is immediately estimated:
$$
\norm{z_T}_{2+\alpha}\le\norm{\dm{G}(\cdot,m_1(T),\cdot)-\dm{G}(\cdot,m_2(T),\cdot)}_{2+\alpha,2+\alpha}\norm{\mu_2(T)}_{-(2+\alpha),D}\le C\dw(m_{01},m_{02})\norm{\mu_0}_{-(2+\alpha)}\,.
$$
As regards the space estimate for $h$, we have
$$
\norm{h(t,\cdot)}_\alpha\le\norm{F'(\cdot,m_1(t),\mu_2(t))-F'(\cdot,m_2(t),\mu_2(t))}_\alpha+\norm{(H'_1-H'_2)(t,\cdot)Dv_2(t,\cdot)}_\alpha\,.
$$
The first term is bounded as $z_T:$
$$
\norm{F'(\cdot,m_1(t),\mu_2(t))-F'(\cdot,m_2(t),\mu_2(t))}_\alpha\le C\dw(m_{01},m_{02})\norm{\mu_0}_{-(2+\alpha)}\,.
$$
The second term, using \eqref{lipsch} and \eqref{sbrigati}, can be estimated in this way:
$$
\norm{(H'_1-H'_2)(t,\cdot)Dv_2(t,\cdot)}_\alpha\le C\norm{(u_1-u_2)(t)}_{1+\alpha}\norm{v_2(t)}_{1+\alpha}\le C\dw(m_{01},m_{02})\norm{\mu_0}_{-(2+\alpha)}\,.
$$
In summary,
$$
\norm{h}_{0,\alpha}=\supo\norm{h(t,\cdot)}_\alpha\le C\dw(m_{01},m_{02})\norm{\mu_0}_{-(2+\alpha)}\,.
$$
Finally, we estimate $\norm{c}_{L^1}$. We have
\begin{align*}
\norm{c}_{L^1} &=\intif(H'_1-H'_2)(t,x)\,\mu_2(t,dx)\,dt+\intif H''_1(t,x)Dv_2(t,x)(m_1(t)-m_2(t))(dx)\,dt\\
&+\intif(H'_1-H'_2)(t,x)\,Dv_2(t,x)\,m_2(t,dx)\,dt\le C\norm{u_1-u_2}\amu\norm{\mu_2}\aam\\
&+C\norm{u_1}\amu\norm{v_2}\amu\norm{m_1-m_2}_{L^1}+C\norm{u_1-u_2}\amu\norm{v_2}\amu\,.
\end{align*}
The first term in the right-hand side, thanks to \eqref{lipsch} and \eqref{sbrigati}, is bounded by
$$
C\norm{u_1-u_2}\amu\norm{\mu_2}\aam\le C\dw(m_{01},m_{02})\norm{\mu_0}_{-(2+\alpha)}\,.
$$
The second and the third term are estimated in the same way, using \eqref{first}, \eqref{lipsch} and \eqref{sbrigati}. Then
$$
\norm{c}_{L^1}\le C\dw(m_{01},m_{02})\norm{\mu_0}_{-(2+\alpha)}\,.
$$
Putting together all these estimates, we finally obtain:
$$
\norm{z}\amv\le C\dw(m_{01},m_{02})\norm{\mu_0}_{-(2+\alpha)}\,.
$$
Since
$$
z(t_0,x)=\into\left(\dm{U}(t_0,x,m_1,y)-\dm{U}(t_0,x,m_2,y)\right)\,\mu_0(dy)\,,
$$
we have proved \eqref{lipsch1}\,.
\end{proof}

We stress the fact that in the bound \eqref{lipsch1} a $2+\alpha$ norm in the last variable is strongly required, to make this estimate valid also for the quantity $D_yD_mU$.

To conclude, we are still left proving the boundary condition for $\dm{U}$ in \eqref{Master} to make true all the hypotheses needed to apply Theorem \ref{settepuntouno}. The latter will be the last result of this section.

\begin{corollary}\label{delarue}
If hypotheses \ref{ipotesi} hold true, then we have the following boundary conditions for $U$:
\begin{equation*}
\begin{split}
&\dm{U}(t,x,m,y)=0\,,\qquad\forall x\in\partial\Omega, y\in\Omega,t\in[0,T],m\in\mathcal{P}^{sub}(\Omega)\,,\\
&\dm{U}(t,x,m,y)=0\,,\qquad\forall x\in\Omega, y\in\partial\Omega,t\in[0,T],m\in\mathcal{P}^{sub}(\Omega)\,.
\end{split}
\end{equation*}
\begin{proof}
The proof of the first boundary condition is trivial, since $\dm{U}(t_0,x,m_0,y)=v(t_0,x;\delta_y)$ and $v$ satisfies a Dirichlet boundary condition.

For the second condition, let us consider $y\in\partial\Omega$. To prove that $v(t_0,x;\delta_y)=0$, it is enough to check that $v=0$ solves 
\begin{equation}\label{muovt}
\begin{cases}
\displaystyle -v_t-\mathrm{tr}(a(x)D^2v)+H_p(t,x,Du)\cdot Dv={\frac{\delta F}{\delta m}}(t,x,m(t))(\mu(t))\,,\\
\displaystyle
v(T,x)={\frac{\delta G}{\delta m}}(x,m(T))(\mu(T))\,,\\
\displaystyle
\bdone{v}\,,
\end{cases}
\end{equation}
where $\mu$ is the unique solution, in the sense of Definition \ref{canzonenuova}, of
$$
\begin{cases}
\mu_t-\nobt{\mu}-\mathrm{div}(\mu H_p(t,x,Du))=0\,,\\
\mu(t_0)=\mu_0\,,\\
\bdone{\mu}\,. 
\end{cases}
$$
In this way we have $mH_{pp}(t,x,Du)Dv=0$, and so the couple $(v,\mu)$ turns out to be a solution of
\begin{equation*}
\begin{cases}
\displaystyle -v_t-\mathrm{tr}(a(x)D^2v)+H_p(t,x,Du)Dv= {\dm{F}}(t,x,m(t))(\rho(t))\,, \medskip \\
\displaystyle	\mu_t-\sum\limits_{ij}\partial^2_{ij}(a_{ij}(x)\mu)  -\mathrm{div}\big(\mu H_p(t,x,Du)+m H_{pp}(t,x,Du) Dv\big)=0 ,\\
\displaystyle v(T,x)= {\dm{G}}(x,m(T))(\mu(T))\,,\qquad\mu(t_0)=\mu_0\,, \medskip\\
\displaystyle \bdone{v}\,,\qquad\bdone{\mu}\, \medskip ,
\end{cases}
\end{equation*}
which is exactly the pure linearized system.

Due to the linearity character of \eqref{muovt}, we only need to prove that
$$
\dm{F}(t,x,m(t))(\mu(t))=\dm{G}(x,m(T))(\mu(T))=0\,.
$$
Thanks to the hypotheses \ref{ipotesi}, both $\dm{F}(x,m(t),\cdot)$ and $\dm{G}(x,m(T),\cdot)$ satisfy a Dirichlet boundary condition, being also elements of $\mathcal{C}^{2+\alpha}$. Then, choosing $\phi(t,y)$ satisfying \eqref{hjbfp} with $\psi(t,y)=0$ and $\xi(y)=\dm{F}(t,x,m(t),y)$, we have

$$	\begin{array}{c}
\displaystyle \dm{F}(t,x,m(t))(\mu(t))=\left\langle \mu(t),\dm{F}(t,x,m(t),\cdot)\right\rangle \medskip \\ \displaystyle \qquad =\langle\mu_0,\phi(0,\cdot)\rangle=\langle \delta_y,\phi(0,\cdot)\rangle=\phi(0,y)=0\,, \medskip
\end{array}
$$
since $\psi$ satisfies a Dirichlet boundary condition. The same computations hold for $\dm{G}$, therefore we have:

\begin{align*}
\dm{U}(t_0,x,m_0,y)=\left\langle\dm{U}(t_0,x,m_0,\cdot),\delta_y\right\rangle=v(t_0,x)=0\,.
\end{align*}
\end{proof}
\end{corollary}






%
%



\section*{Declarations}

\begin{itemize}
\item All authors contributed to the study conception and design. All authors read and approved the final manuscript.
\item No funds, grants, or other support was received.
\item All authors declare that they have no financial interests or conflicts of interest.
\end{itemize}




\section*{Acknowledgements}
We would like to thank Espen R. Jakobsen for the useful suggestions.

\bibliographystyle{abbrv}
\bibliography{MasterEquation_absorption}

\end{document}